\documentclass[12pt,reqno,a4paper]{amsart}

\usepackage{amsmath}
\usepackage{amsfonts}
\usepackage{amsthm}
\usepackage{amstext}
\usepackage{enumerate}
\usepackage[latin1]{inputenc}
\usepackage[shortlabels]{enumitem}

\usepackage{graphicx}
\usepackage{verbatim}

     \newcommand{\supp}{\operatorname{supp}}

\newcommand{\Ker}{{\operatorname{Ker}}}

     \newcommand{\N}{{\mathbb{N}}}

     \newcommand{\R}{{\mathbb{R}}}
     \newcommand{\Z}{{\mathbb{Z}}}
     \newcommand{\C}{{\mathbb{C}}}

\newcommand{\e}{{\rm e}}

\renewcommand{\i}{{\rm i}}
\renewcommand{\d}{{\rm d}}

\newcommand{\per}{{\rm per}}

\newcommand{\sr}{{\rm sr}}

\renewcommand{\Re}{{\rm Re}\,}
\renewcommand{\Im}{{\rm Im}\,}

\newcommand\inp[2][]{#1 \langle #2#1\rangle}
\newcommand\parb[2][]{#1 \big ( #2#1\big )}
\newcommand\parbb[2][]{#1 \Big ( #2#1\Big )}

\newcommand{\phy}{{\rm phy}}

\newcommand{\pp}{{\rm pp}}

\newcommand{\mand}{\text{ and }}
\newcommand{\mfor}{\text{ for }}
\newcommand{\mforall}{\text{ for all }}

\newcommand{\w}[1]{\langle {#1} \rangle}

\newcommand{\bN}{{\mathbb N}}

\newcommand{\bH}{{\mathbb H}}

\newcommand{\bS}{{\mathbb S}}

\newcommand{\vB}{{\mathcal B}}

\newcommand{\vD}{{\mathcal D}}

\newcommand{\vH}{{\mathcal H}}

     \theoremstyle{plain}
     \newtheorem{thm}{Theorem}[section]
     \newtheorem{prop}[thm]{Proposition}
     \newtheorem{lemma}[thm]{Lemma}
      \newtheorem{cor}[thm]{Corollary}
     \theoremstyle{definition}
     
     \newtheorem{defn}[thm]{Definition}

     \newtheorem{cond}[thm]{Condition}

\newtheorem*{remarks*}{Remarks}
\newtheorem*{remark*}{Remark}

\numberwithin{equation}{section}

\makeatother

\begin{document}

\title[Threshold spectral analysis]{Two-body threshold spectral analysis, the critical case}

\author{Erik Skibsted}
\address[Erik Skibsted]{Institut for  Matematiske
Fag \\
Aarhus Universitet\\ Ny Munkegade  8000 Aarhus C,
Denmark}
\email{skibsted@imf.au.dk}

\author{Xue Ping Wang}
\address[Xue Ping  Wang]{Laboratoire de Mathématiques Jean Leray, UMR CNRS 6629\\
Université de Nantes\\
44322 Nantes Cedex, France}
\email{xue-ping.wang@univ-nantes.fr}

\thanks{X.P. W. is  supported in part by the French National Research Agency  under the project No. ANR-08-BLAN-0228-01}
\subjclass{35P25, 47A40, 81U10} 
\keywords{Threshold spectral analysis, Schr\"odinger operator, critical potential, phase shift}

\begin{abstract}
We study in dimension $d\geq2$ low-energy spectral and scattering
asymptotics for two-body $d$-dimensional Schr\"odinger
operators with a radially symmetric potential falling off like
$-\gamma r^{-2},\;\gamma>0$. We consider angular momentum sectors, labelled by
$l=0,1,\dots$,  for
which $\gamma>(l+d/2
-1)^2$. In each such sector the reduced Schr\"odinger
operator has infinitely many negative eigenvalues accumulating
at zero. We show that the resolvent has a non-trivial
oscillatory behaviour as the spectral parameter approaches zero  in
cones bounded away from the negative half-axis, and we
derive an  asymptotic formula for the phase shift.
\end{abstract}

\maketitle
\tableofcontents

\section{Introduction}

The low-energy spectral and scattering asymptotics for two-body Schr\"odinger
operators depends heavily on the decay of the potential at
infinity. The most well-studied class is given by   potentials 
decaying  faster  than $r^{-2}$ (see for example \cite{JN} and references there). The
expansion  of the resolvent is in this case in  terms of powers of  dimension-dependent modifications of the spectral parameter
and it depends on possible existence of zero-energy bound states and/or zero-energy
resonance states. Classes of negative potentials  decaying slower than
$r^{-2}$ were  studied in
\cite{FS, Ya}. In that case the resolvent is more regular
at zero energy. It has an expansion in integer powers of the spectral parameter and there are no zero-energy bound
states nor resonance states. Moreover, the nature of the expansion is ``semi-classical''. 
For general perturbations of critical decay of the order $r^{-2}$  and
with 
an assumption related to  the Hardy inequality, the
threshold spectral analysis is  carried out in \cite{Wan1,Wan2}. It is
shown that for this class of potentials  the zero resonance may appear in any space dimension with arbitrary multiplicity. Recall that for potentials decaying faster than $r^{-2}$,  the zero resonance is absent if the space dimension $d$ is bigger than or equal to five and its multiplicity is at most one when $d$ is equal to three or four.   The goal of this paper is to treat 
a class of radially symmetric potentials decaying
like $-\gamma r^{-2}$ at infinity, where $\gamma >0$ is big such that the condition used in \cite{Wan1,Wan2} is not satisfied. In this case,  there exist infinitely
many negative eigenvalues (see \eqref{eq:85p} for a precise
condition). We will give a resolvent expansion as well as
an asymptotic formula for the phase shift. These
expansions are to our knowledge not  semi-classical even though there
are common features with the more slowly decaying case. 

Consider for  $d\geq 2$ the $d$-dimensional
Schr\"odinger operator
\begin{equation*}
  Hv=(-\triangle +W)v=0,
\end{equation*}
 for a radial
potential $W=W(|x|)$ obeying 

\begin{cond}\label{cond:asympt-full-hamilt2bp}
 \begin{enumerate}[1)]
\item \label{item:4p}$W(r)=W_1(r)+W_2(r);\;W_1(r)=-\tfrac{\gamma}{r^2}\chi(r>1)\mfor\text{
  some }
   \gamma>0$,
  \item \label{item:1bbp}
  $W_2\in C(]0,\infty[,\R)$,
\item \label{item:2bbp} $\exists\, \epsilon_1,C_1>0:\;|W_2(r)|\leq
  C_1r^{-2-\epsilon_1}\mfor r>1$,
\item \label{item:3bbp} $\exists\, \epsilon_2,C_2>0:
  |W_2(r)|\leq C_2r^{\epsilon_2-2}\mfor r\leq  1$.
  \end{enumerate} 
\end{cond}

 Here the function $\chi(r >1)$ is a 
smooth cutoff function taken to be $1$ for $r\geq 2$ and $0$ for $r\leq 1$
(see the end of this introduction for the precise 
definition). 
 Under Condition \ref{cond:asympt-full-hamilt2bp} $H$ is self-adjoint
 as defined in terms of the Dirichlet form on $H^1(\R^d)$. Let  $H_l$,
 $l=0,1,\dots$, be the  corresponding 
reduced  Hamiltonian on $L^2(\R_+)$ corresponding to the  eigenvalue
$l(l+d-2)$ of
the Laplace-Beltrami operator on $S^{d-1}$ 
\begin{equation}
  \label{eq:80p}
  H_lu=-u''+(V_\infty+V)u.
\end{equation} Here
\begin{subequations}
  \begin{align}
  \label{eq:83p}
 V_\infty(r)&= \tfrac{\nu^2-1/4}{r^2}\chi(r>1);\; \nu^2=(l+\tfrac d2
-1)^2-\gamma,\\
V(r)&=W_2(r)+ \tfrac{(l+\tfrac d2
-1)^2-1/4}{r^2}\big (1-\chi(r>1)\big ).\label{eq:84p} 
\end{align}
\end{subequations}
 Notice that $V$ is small at infinity compared to
$V_\infty$. We are interested in spectral and scattering properties of
$H_l$ at zero energy in the case
\begin{equation}
  \label{eq:85p}
  \gamma>(l+\tfrac d2
-1)^2.
\end{equation}  
This condition is equivalent to having $\nu$ in
\eqref{eq:83p} purely imaginary (for convenience we fix it in this
case as $\nu=-\i \sigma,
\, \sigma>0$), and it implies the existence of a
sequence of negative eigenvalues of $H_l$ accumulating at zero
energy. 

Our first main result is on the expansion of the resolvent
\begin{equation*}
    R_l(k):=(H_l-k^2)^{-1}\mfor\, k\in \Gamma_{\theta}^{\pm},
  \end{equation*} where here (for any $\theta\in ]0,\pi/2[$)
\begin{subequations}
 \begin{align*}
  \Gamma_{\theta}^+&= \{ k\neq 0|\, 0 < \arg k \leq \theta\},\\
\Gamma_{\theta}^{-}&=\{ k\neq 0|\, \pi-\theta
  \leq \arg k < \pi\}.
\end{align*} 
  \end{subequations} We say that a solution $u$ to the equation
  \begin{equation}\label{eq:50p}
   -u''(r)+ \parb{V_\infty(r)+V(r)}u(r)=0
  \end{equation} is   {\it regular} if the
  function $r\to \chi(r<1)u(r)$ belongs to $\vD (H_l)$. For any $t\in
  \R$ we introduce the weighted $L^2$-space  $\vH_t:=\inp{r}^{-t}L^2(\R_+);\,\inp{r}=(1+r^2)^{1/2}$. 
  
\begin{thm} \label{thm1} Suppose Condition \ref{cond:asympt-full-hamilt2bp} and
  \eqref{eq:85p} for some (fixed) $l\in\N\cup\{0\}$. Let $\theta\in ]0,\pi/2[$. 
    \label{thm:constr-resolvp} There exist (finite) rational
    functions $f^{\pm}$ in the variable  $k^{2\nu}$ for $k\in
    \Gamma_{\theta}^{\pm}$ for which
    \begin{equation}\label{eq:45p}
      \lim_{\Gamma_{\theta}^{\pm}\ni k\to
        0}\Im f^{\pm}(k^{2\nu})\text{ do not exist,}
    \end{equation}  
there exist  Green's functions for $H_l$ at zero
energy, denoted $R_0^{\pm}$,  and there  exists a real nonzero regular
solution to (\ref{eq:50p}), denoted $u$,  such that the following
    asymptotics hold. For all $s>s'>1,\,s\leq 1+\epsilon_1/2,\,s'\leq 3$:
 \begin{equation}
  \limsup _{\Gamma_{\theta}^{\pm}\ni k\to
        0}\;|k|^{1-s'}\big \|R_l(k)-R_0^{\pm}-f^{\pm}(k^{2\nu})
  |u\rangle \langle u|\big\|_{\vB(\vH_s,\vH_{-s})} <\infty.\label{eq:24bbbbp}
\end{equation} 
\end {thm}

Due to \eqref{eq:45p} the rank-one operators $f^{\pm}(k^{2\nu})
  |u\rangle \langle u|$ in \eqref{eq:24bbbbp} are non-trivially oscillatory. This
  phenomenon does not occur  for low-energy resolvent expansions for
  potentials either decaying faster or slower than $r^{-2}$
  (cf. \cite{JN} and \cite{FS, Ya}, respectively), nor  for sectors where   \eqref{eq:85p} is
  not fulfilled  (cf. \cite{Wan2}).   Combining Theorem \ref{thm1} and the results of \cite{Wan2}, we can  deduce the resolvent asymptotics near threshold for $d$-dimensional Schr\"odinger operators with
  critically decaying,  spherically symmetric potentials, see Theorem \ref{thm3.7}. An advantage to work with spherically symmetric potentials is that we can diagonalize the operator in spherical harmonics and explicitly calculate some subtle quantities. For example, one can easily show that if zero is a resonance of $H$, then its multiplicity is  equal to
\[
\frac{ (m+d-3)!}{(d-2)! (m-1)!}  + \frac{ (m+d-2)!}{(d-2)! m!}
\]
where $m \in\bN\cup\{0\}$ is such that  $ (m+ \frac d 2 -1)^2 -\gamma \in ]0, 1]$.  This shows that multiplicity of zero resonance grows like $\gamma^{\frac{d-2}{2}}$ when $\gamma$ is big and
$d\ge 3$. To study the resolvent asymptotics for non-spherically
symmetric potential $W(x)$  behaving like $\frac{q(\theta)}{r^2}$  at
infinity ($ x= r\theta$ with $r =|x|$), one is led  to analyze the
interactions  between different oscillations and resonant states. This
is not carried out in the present  work. 

Our second main result is on the asymptotics of the phase shift. Let  $u_l$ be a regular
solution to the reduced Schr\"odinger equation 
\begin{equation*}
 -u''+(V_\infty+V)u=\lambda u;\;\lambda>0. 
\end{equation*} Write 
\begin{equation*}
  \lim_{r\to\infty}
\left(u_l(r)-C 
\sin \big(\sqrt \lambda r+D_l\big
)\right)
=0. 
\end{equation*} 
The standard definition of the
phase shift (coinciding with the time-depending definition) is
 \begin{equation*}
   \sigma_l^\phy(\lambda)=D_l+\tfrac{d-3+2l}{4}\pi.
 \end{equation*}

The notation 
    $\sigma^{\per}=\sigma^{\per}(t)$ signifies  below the continuous real-valued $2\pi$-periodic function
    determined by
\begin{equation*}
\begin{cases}
\sigma^{\per}(0)&=0\\
\e^{\pi\sigma}\e^{-\i t}-\e^{\i t}&=r(t)\e^{\i (\sigma^{\per}(t)-t)};\; \;r(t)>0, t\in \R.
\end{cases}\;.
\end{equation*}

\begin{thm}
    \label{thm:phasebbp} Suppose Condition
    \ref{cond:asympt-full-hamilt2bp} and \eqref{eq:85p} for some
    $l\in\N\cup\{0\}$. Let
    \begin{equation*}
      \sigma=\sqrt{\gamma-(l+\tfrac d2 -1)^2}
    \end{equation*} 
(recall $\nu=-\i\sigma$).
    There exist $C_1,C_2\in \R$ such that
    \begin{equation}
      \label{eq:64vvp}
     \sigma^{\phy}_l(\lambda)+\sigma\ln \sqrt{\lambda}- \sigma^{\per}( \sigma\ln \sqrt{\lambda}+C_1)\to 
C_2\mfor \lambda\downarrow 0.    
\end{equation} 
\end{thm}

Whence the leading term in the asymptotics of the phase
shift is linear in $\ln \sqrt{\lambda}$ while the next term is
oscillatory in the same quantity. The (positive) sign agrees with the
well-known Levinson
theorem (cf. \cite[(12.95) and (12.156)]{Ne}) valid for potentials 
decaying faster than $r^{-2}$. Also the qualitative
behaviour of these terms as $\sigma\to 0$  (i.e. finiteness in the
limit) is agreeable to the case
where \eqref{eq:85p} is not fulfilled  (studied in  \cite{Ca} from a different point of
view). 

The bulk of this paper concerns  somewhat more general one-dimensional problems
than discussed above. In particular  we consider for  $(d,l)\neq
(2,0)$ a model with a local singularity at $r=0$ that is more general than
specified by Condition \ref{cond:asympt-full-hamilt2bp}
\ref{item:3bbp} and \eqref{eq:84p}. This extension does not contribute
by any complication and is therefore naturally included. It would be  possible to extend
our methods to certain types of more general local singularities,
however this would add some extra complication that we will not 
pursue. Our methods rely heavily on explicit properties of solutions to the
Bessel equation as well as    ODE techniques. These properties
 compensate for the fact that, at least to our knowledge, semi-classical
analysis is not doable in the present context (for
instance the semi-classical formula \eqref{eq:relat} for the
asymptotics of the phase shift for slowly
decaying potentials is not correct under Condition
\ref{cond:asympt-full-hamilt2bp}.) See however \cite{CSST} in the case the potential is positive.

One of our motivations for studying a potential with  critical fall off comes from
an $N$-body problem: Consider a $2$-cluster $N$-body threshold under the
assumption of Coulomb pair interactions, this could be given by two
atoms each one being confined in a bound state. Suppose one atom is
charged while the other one is neutral. The effective intercluster
potential will in this case in a typical situation (given by nonzero 
moment of charge of the 
bound state of the neutral atom) have $r^{-2}$ decay
although with some angular dependence (the so-called dipole approximation). Whence  we expect (due to the present
work) that the  $N$-body resolvent will have some oscillatory
behaviour  near the threshold in question. Proving this (and related spectral
 and scattering properties) would,  in addition to material from the present paper, rely on a
reduction scheme not to be discussed here. We plan to study this
problem in a separate  future  publication.

In this  paper we consider parameters  $\pm\nu, z\in \C$ satisfying
 $\nu=-\i \sigma$ where $\sigma >0$  and $z\in \C\setminus
\{0\}$ with  $\Im z\geq 0$. Powers of $z$ are throughout the paper defined in terms
of  the argument function fixed by the condition $\arg z \in
[0,\pi]$.  We shall use the standard notation
$\inp{z}:=(1+|z|^2)^{1/2}$. For any given $c>0$ we shall use the notation $\chi(r>c)$ to denote  a  given
real-valued  function $\chi\in C^\infty(\R_+)$ with $\chi(r)=0$ for $r\leq c$ and
$\chi(r)=1$ for $r\geq 2c$. We take it such that there exists a  real-valued  
function $\chi_<\in C^\infty(\R_+)$, denoted $\chi_<=\chi(\cdot<c)$,
such that $\chi^2+\chi_<^2=1$. Let for $\theta \in [0,\pi/2[$ and $\epsilon>0$ 
\begin{align}
  \label{eq:11}
  \Gamma_{\theta,\epsilon}&= \{ k\neq 0|\, 0 \leq \arg k \leq \theta\text{ or }\pi-\theta
  \leq \arg k \leq \pi\}\cap\{|k|\leq \epsilon\},\\
\Gamma_{\theta,\epsilon}^{\pm}&=\Gamma_{\theta,\epsilon}\cap \{\pm \Re k>0\}.\label{eq:30}
 \end{align}

\bigskip

\section{Model asymptotics} \label{Model asymptotics}
 
In this Section, we give the resolvent asymptotics at zero for a model operator under the condition (\ref{eq:85p}). See \cite{Wan1}) when (\ref{eq:85p}) is not satisfied.  Recall firstly some
basic formulas for Bessel and Hankel functions from \cite[pp. 228--230]{Ta1} and \cite[pp. 126--127,
204]{Ta2} (or see  \cite{Wat}):
\begin{subequations}
  \label{eq:57sub}
\begin{align}
  \label{eq:1}
  J_\nu(z)&=\tfrac{(z/2)^{\nu}}{\Gamma(1/2)\Gamma(\nu +1/2)}\int_{-1}^{1}(1-t^2)^{\nu-1/2}\e^{\i zt}\,\d t,\\
\label{eq:2}\int_{-1}^{1}(1-t^2)^{\nu-1/2}\,\d
t&=\tfrac{\Gamma(1/2)\Gamma(\nu +1/2)}{\Gamma(\nu+1)},\\
\label{eq:3}H^{(1)}_\nu(z)&=\frac{ J_{-\nu}(z)-\e^{-\i \nu \pi}
  J_{\nu}(z)}{\i \sin (\nu\pi)},\\
\label{eq:4}H^{(1)}_\nu(z)&=\parb{\tfrac {2}{\pi
    z}}^{1/2}\tfrac{\e^{\i (z-\nu \pi/2-\pi/4)}}{\Gamma(\nu +1/2)}
\int_{0}^{\infty}\e^{-t}t^{\nu-1/2}(1-\tfrac{t}{2\i z})^{\nu-1/2}\,\d t.
\end{align}
\end{subequations}
The functions $J_{\nu}$ and  $H^{(1)}_\nu$ solve the Bessel equation
\begin{equation}
  \label{eq:5a}
  z^{-1/2}\parb{-\tfrac{\d^2}{\d z^2}+\tfrac{\nu^2-1/4}{z^2}-1}z^{1/2}u(z)=0.
\end{equation}

We have
\begin{subequations}
  \label{eq:58sub}
\begin{align}\label{eq:5} 
J_\nu(z)&=\e ^{\i \nu \pi}\overline
{ J_{\bar \nu}(-\bar z)},\\
H^{(1)}_\nu(z)&=  \e ^{- \i \nu \pi}H^{(1)}_{-\nu}(z)=-\overline
{ H^{(1)}_{-\bar  \nu}(-\bar z)}.\label{eq:6}
\end{align}
\end{subequations}

\subsection{Model operator and construction of model resolvent}
Consider 
\begin{equation}
  \label{eq:25}
  H^D=-\tfrac{\d^2}{\d r^2}+\tfrac{\nu^2-1/4}{r^2}\text { on }\vH^D
=L^2([1, \infty[) 
\end{equation}
 with Dirichlet boundary condition at $r=1$. Let
 for any $\zeta\in \C$, $\phi=\phi_\zeta$ be the (unique) solution to 
\begin{equation}\label{eq:equamotionBA}
\begin{cases}
-\phi''(r)+\tfrac{\nu^2-1/4}{r^2}\phi(r)&=\zeta \,\phi(r)\\
\phi(1)&=0\\
\phi'(1)&=1
\end{cases}\;.
\end{equation}
  This solution $\phi_\zeta$ is entire in $\zeta$, and 
  \begin{equation}
    \label{eq:7}
    \phi_0(r)=\frac {r^{1/2+\nu}-r^{1/2-\nu}}{2\nu}.
  \end{equation}

In fact, cf. \cite[(3.6.27)]{Ta1},
\begin{equation}
  \label{eq:20}
  \phi_{k^2}(r)=\tfrac{\pi}{2\sin (\nu \pi )}r^{1/2}
\parb{J_{\bar \nu}(k)J_{\nu}(kr)-J_{\nu}(k)J_{\bar \nu}(kr)}.
\end{equation}

Let for $k\in \C\setminus
\{0\}$ with  $\Im k\geq 0$ and $H^{(1)}_\nu(k)\neq 0$
\begin{equation}
  \label{eq:8}
  \phi^+_k(r)=r^{1/2}\frac {H^{(1)}_\nu(kr)}{H^{(1)}_\nu(k)}.
\end{equation} 

Due to (\ref{eq:6}) the dependence of $\nu$ in $\phi^+_k$ is through
$\nu^2$ only, i.e. replacing $\nu\to \bar\nu$ yields the same expression  (obviously this is also true for $\phi_{k^2}$). 
Notice also that $\phi_{k^2}$  and $\phi^+_k$ solve 
the equation 
\begin{equation}
  \label{eq:9}
  -\phi''(r)+\tfrac{\nu^2-1/4}{r^2}\phi(r)=k^2 \,\phi(r).
\end{equation} The kernel $R^D_k(r,r')$ of $(H^D-k^2)^{-1}$ for $k$ with  $\Im k> 0$ and $H^{(1)}_\nu(k)\neq 0$ is given by
\begin{equation}
  \label{eq:10}
  R^D_k(r,r')=\phi_{k^2}(r_<)\phi^+_k(r_>);
\end{equation} here and henceforth $\;r_<:=\min (r,r')\mand r_>:=\max
(r,r')$. (The fact that the right hand side of (\ref{eq:10}) defines a
bounded operator on $\vH^D$ follows from the Schur test and the bounds (\ref{eq:18aaa}) and
(\ref{eq:18bdd}) given below.)
 The condition $H^{(1)}_\nu(k)\neq 0$ is fulfilled for $k\in\{\Im k>0\}\setminus
\i \R_+$ since otherwise $k^2$ would be a non-real eigenvalue of
$H^D$. The zeros in $\i \R_+$ correspond to the negative eigenvalues
of $H^D$. They constitute  a sequence  accumulating at zero.

We have the properties, cf. (\ref{eq:6}),
\begin{equation}
  \label{eq:12}
 {R^D_k(r,r')}=  \overline {R^D_{-\bar k}(r,r')}=R^D_k(r',r).
\end{equation}
In the regime where $|k|$ is very small and stays away from the
imaginary axis, more precisely in $\Gamma_{\theta,\epsilon}$ for any $\theta \in [0,\pi/2[$ and $\epsilon>0$,  we can derive a lower bound of $|H^{(1)}_\nu(k)|$ as
follows:
From (\ref{eq:1}) and (\ref{eq:2}) we obtain that
\begin{equation}
  \label{eq:13}
 J_\nu(z)= \tfrac{(z/2)^{\nu}}{\Gamma(\nu+1)}\big (1+O(z^2)\big ).
\end{equation} Whence (recall that   $\nu=-\i \sigma$ where
$\sigma>0$)  we obtain
with $C_\nu:=|\Gamma(\nu +1)\sin (\nu\pi)|$
\begin{align}
 |H^{(1)}_\nu(k)|& \geq \parb{|\e^{-\sigma \arg k}-\e^{-\sigma
      \pi}\e^{\sigma \arg k}|-O(|k|^2)}/C_\nu\nonumber\\
& \geq \e^{-\sigma
      \theta}\parb{1-\e^{-\sigma
      (\pi-2\theta)}}/C_\nu-O(|k|^2)\mforall k\in \Gamma_{\theta,\epsilon}.
 \label{eq:14} 
\end{align} 
In particular for $\epsilon>0$ small enough (depending on $\theta$)
\begin{equation}
  \label{eq:15}
 \forall  k\in \Gamma_{\theta,\epsilon}:\, |H^{(1)}_\nu(k)|\geq \e^{-\sigma
      \pi/2}\parb{1-\e^{-\sigma
      (\pi-2\theta)}}/C_\nu.
\end{equation} 

Note that the bound (\ref{eq:15}) implies that there is a limiting
absorption principle at all real $E=k^2$ with $k\in
\Gamma_{\theta,\epsilon}$. In particular $H^D$ does not have small
positive eigenvalues.

\subsection{Asymptotics of model resolvent}

Let us note the following global bound (cf. (\ref{eq:4})) 
\begin{equation}
  \label{eq:18aaa}
 |\phi^+_k(r)|\leq  C\parb{\tfrac{r}{\inp{kr}}}^{1/2}\e
 ^{-(\Im k) r}\mforall k\in \Gamma_{\theta,\epsilon}\mand r\geq 1.
\end{equation}

Let
\begin{equation}
  \label{eq:16}
  D_\nu=2^{-\nu}/\Gamma(\nu +1).
\end{equation} Notice that $\bar D_\nu=D_{-\nu}$. By (\ref{eq:3}) and (\ref{eq:13}) we obtain the
following asymptotics of $\phi^+_k$ as $k\to 0$ in $\Gamma_{\theta,\epsilon} $:
\begin{equation}
  \label{eq:17}
 \phi^+_k(r)= r^{1/2}\frac{\bar D_\nu r^{-\nu}k^{-\nu}-\e^{-\sigma\pi}D_\nu r^{\nu}k^{\nu}+O\parb{(kr)^2}} {\bar D_\nu k^{-\nu}-\e^{-\sigma\pi}D_\nu k^{\nu}+O\parb{k^2}}.
\end{equation}

Introducing
\begin{equation}
  \label{eq:18ii}
  \zeta(k)=\frac{2\i \sigma \e^{-\sigma\pi}D_\nu k^{2\nu}}{\bar D_\nu-D_\nu\e^{-\sigma\pi}k^{2\nu}},
\end{equation} we can slightly modify  (\ref{eq:17}) (in terms of (\ref{eq:7})
and by using (\ref{eq:18aaa}))  as 
\begin{equation}
  \label{eq:18}
 \phi^+_k(r)= r^{1/2-\nu}+\zeta(k)\phi_0(r)+ r^{1/2}O\parb{(kr)^2} +\parb{\tfrac{r}{\inp{kr}}}^{1/2}\e
 ^{-(\Im k) r} O\parb{k^2}.
\end{equation} There is a ``global''  bound of the third term  (due to  (\ref{eq:18aaa})):
\begin{equation}
  \label{eq:22ii}
  |
r^{1/2}O\parb{(kr)^2}|\leq Cr^{1/2}\tfrac{|kr|^2}{\inp{kr}^{2}}\mforall k\in \Gamma_{\theta,\epsilon}\mand r\geq 1.
\end{equation}

As for $\phi_{k^2}$ we first note 
 the following global bound (cf. (\ref{eq:1}), (\ref{eq:20}) and
 \cite[Theorem 4.6.1]{Ol})
\begin{equation}
  \label{eq:18bdd}
 |\phi_{k^2}(r)|\leq C\parb{\tfrac{r}{\inp{kr}}}^{1/2}\e
 ^{(\Im k) r}\mforall k\in \Gamma_{\theta,\epsilon}\mand r\geq 1.
\end{equation}

 Using (\ref{eq:18bdd})  we obtain similarly
\begin{equation}
  \label{eq:18b}
 \phi_{k^2}(r)= \phi_{0}(r)+ r^{1/2}O\parb{(kr)^2}
 +\parb{\tfrac{r}{\inp{kr}}}^{1/2}\e
 ^{(\Im k) r} O\parb{k^2}.
\end{equation}  
There is a global bound of the second  term:
\begin{equation}
  \label{eq:22}
  |
r^{1/2}O\parb{(kr)^2}|\leq Cr^{1/2}\tfrac{|kr|^2}{\inp{kr}^{2}}\e
 ^{(\Im k) r} \mforall k\in \Gamma_{\theta,\epsilon}\mand r\geq 1.
\end{equation}

Whence in combination with (\ref{eq:10}) we obtain uniformly in $k\in
\Gamma_{\theta,\epsilon}$ and $r, r'\geq 1$
\begin{subequations}
 \begin{align}
  \label{eq:19}
  R^D_k(r,r')&=R_0^D(r,r')+\zeta(k) T(r,r')+r^{1/2}(r')^{1/2}E_k(r,r');\\
R_0^D(r,r') &=\phi_{0}(r_<)r_>^{1/2-\nu },\label{eq:19i}\\
T(r,r') &=\phi_{0}(r) \phi_{0}(r'),\label{eq:19ii}\\
|E_k(r,r')| &\leq C \Big (
  \frac {|k|r_>}{\inp{kr_>}}\Big )^2.\label{eq:19iii}
\end{align} 
\end{subequations}
 
Clearly $T=|\phi_{0}\rangle \langle \phi_{0}|$ is a
rank-one  operator and the function $\zeta$ has a non-trivial oscillatory
behaviour. The error estimate can be replaced by:
\begin{equation}
  \label{eq:21}
  \exists C>0\,\forall \delta\in[0,2]:|E_k(r,r')| \leq C |kr_>|^\delta
  \mforall k\in
  \Gamma_{\theta,\epsilon}\mand r, r'\geq 1.
\end{equation} In particular introducing weighted spaces
\begin{equation*}
  \vH_s^D=\inp{r}^{-s}\vH^D,
\end{equation*}
 we obtain
\begin{equation}
  \label{eq:23}
  \forall \,s>1: \;\lim _{\Gamma_{\theta,\epsilon}\ni k\to 0}\big \|R^D_k-R_0^D-\zeta(k)
  T\big \|_{\vB(\vH_s^D,\vH_{-s}^D)}=0.
\end{equation} 
In fact we deduce from (\ref{eq:19})-(\ref{eq:19iii}) the following more precise result:
 \begin{lemma}
   \label{lemma:asympt-model-resolv} For all $s>s'>1,\,s'\leq 3$, there
   exists $C>0$:
 \begin{equation}
  \label{eq:24}
  \big \|(H^D-\i )\parb{R^D_k-R_0^D-\zeta(k)
  T}\big \|_{\vB(\vH_s^D,\vH_{-s}^D)} \leq C|k|^{s'-1}\mforall k\in \Gamma_{\theta,\epsilon}.
\end{equation}
\end{lemma}

\bigskip

\section{Asymptotics for full Hamiltonian, compactly supported perturbation} \label{Asymptotics for full Hamiltonian, compactly supported perturbation}

Consider with $V_\infty(r):=\tfrac{\nu^2-1/4}{r^2}\chi(r>1)$
\begin{equation}
  \label{eq:25c}
  H=-\tfrac{\d^2}{\d r^2}+V_\infty+V\text { on }\vH
:=L^2(]0, \infty[)
\end{equation}  with Dirichlet boundary condition at $r=0$. As for the
potential $V$ we impose in this section 

\begin{cond}\label{cond:asympt-full-hamilt}
 \begin{enumerate}[1)]
  \item \label{item:1}
  $V\in C(]0,\infty[,\R)$,
\item \label{item:2} $\exists R>3:\, V(r)=0 \mfor r\geq R$,

\item \label{item:3} $\exists C_1,C_2>0\,\exists \kappa>0:
  C_1(r^{-2}+1)\geq V(r)\geq (\kappa^2-1/4)r^{-2}-C_2$.
  \end{enumerate} 
\end{cond}
Notice that  the operator $H$ is defined in terms of the (closed) 
Dirichlet form on the Sobolev space  $H^1_0(\R_+)$ (i.e. $H$   is the
Friedrichs extension), cf. \cite[Lemma 5.3.1]{Da}. For the limiting
cases  $C_1=\infty$ and/or
$\kappa=0$ in \ref{item:3} it is still possible to define $H$ as the
Friedrichs extension of the action on   $C_c^{\infty}(\R_+)$ however
the form domain of the extension might be different from $H^1_0(\R_+)$ and
some arguments of this paper would be  more complicated. An example of this type (with $\kappa=0$) is discussed
in Appendix \ref{Case (d,l)eq (2,0)}.
 If $V(r)\geq 3/4r^{-2}-C$ the operator $H$ is
  essentially self-adjoint on $C_c^{\infty}(\R_+)$, cf. \cite[Theorem
  X.10]{RS}.

In terms of  the resolvent $R_k^D$ considered in Section \ref{Model
  asymptotics} and cutoffs $\chi_1=\chi_1(r<7)$ and
$\chi_2=\chi_2(r>7)$ we introduce for
$k\in \Gamma_{\theta,\epsilon}$ 
\begin{equation}
  \label{eq:27}
  G_k=\chi_1\big (H-\tfrac{\Re k}{|\Re k|}\i\big )^{-1}\chi_1+\chi_2R_k^D\chi_2.
\end{equation}

Let 
\begin{equation}
  \label{eq:28}
  G_0^{\pm}=\chi_1\big (H\mp\i\big )^{-1}\chi_1+\chi_2R_0^D\chi_2
\end{equation}  and
\begin{equation}
  \label{eq:29}
  K^{\pm}=HG_0^{\pm}-I.
\end{equation} Notice that the operators $K^{\pm}$ are compact on 
  $\vH_s:=\inp{r}^{-s}\vH$ for $s>1$.

Due to Lemma
  \ref{lemma:asympt-model-resolv} we have  the following expansions
  in $\vB(\vH_s)$ (with $s>s'>1$ and $\,s'\leq 3$)
  \begin{equation}
    \label{eq:26}
    \forall k\in \Gamma_{\theta,\epsilon}^{\pm}:\,(H-k^2)G_k=I+ K^{\pm}+\zeta(k)|\psi_0\rangle \langle \chi_2\phi_0|+O\big (|k|^{s'-1}\big );\;\psi_0:=H\chi_2\phi_0.
  \end{equation}

  \begin{lemma}
    \label{lemma:asympt-full-hamilt}
For all $k\in \Gamma_{\theta,\epsilon}^{\pm}$ the following form
inequality holds (on $\vH_s$ for any $s> 1$)
\begin{equation}
  \label{eq:31}
  \pm \Im G_k \geq \chi_1\big (H\pm\i\big )^{-1}\big (H\mp\i\big )^{-1}\chi_1.
\end{equation}
\end{lemma}
\begin{proof}
  This is obvious from the fact that  $\pm \Im R_k^D\geq 0$.
\end{proof}

  \begin{prop} \label{prop:no_kernel} For all $s>1$ the operators  $I+ K^{\pm}\in\vB(\vH_s)$
    have   zero null space, i.e.
    \begin{equation}
      \label{eq:32}
      \Ker \parb{I+ K^{\pm}}=\{0\}.
    \end{equation}
\end{prop}
\begin{proof} We prove only (\ref{eq:32}) for the superscript ``$+$
  case''. The ``$-$ case'' is similar.   Suppose
  $0=HG_0^+f$ for some $f\in \vH_s$. We shall show that $f=0$. Let
  $u_0=G_0^+f$. Integrating by parts yields
  \begin{align}
    \label{eq:33}
     0&=\Im \inp{u_0,-Hu_0}=
\lim_{r\to \infty}\Im \parb{\bar u_0 u'_0}(r)\nonumber\\& =
\lim_{r\to \infty}\Im \parb{(1/2-\nu)|u_0|^2(r)/r}=\sigma |\inp{\chi_2\phi_0,f}|^2.
 \end{align}
 
So
\begin{equation}
  \label{eq:34}
  \inp{\chi_2\phi_0,f}=0,
\end{equation} and therefore (seen again  by using the explicit
kernel of $R^D_0$ and by estimating by the Cauchy-Schwarz inequality)
\begin{equation}
  \label{eq:35}
  u_0=O(r^{3/2-s})\mand u_0'=O(r^{1/2-s})\mfor r\to \infty.
\end{equation} From (\ref{eq:35}) we can conclude that 
\begin{equation}
  \label{eq:36}
   u_0=0;
\end{equation} this can be seen by writing $u_0$
as a linear combination of $r^{1/2+\nu}$ and $r^{1/2-\nu}$ at
infinity, deduce that $u_0$ vanishes at infinity
and then  invoke unique continuation. For a more general result (with
detailed proof) see Lemma \ref{lemma_bound}.

Using Lemmas \ref{lemma:asympt-model-resolv} and \ref{lemma:asympt-full-hamilt}, (\ref{eq:34}) and
(\ref{eq:36}) 
we compute
\begin{equation}
  0=\Im \inp{f,u_0}=\lim_{\Gamma_{\theta,\epsilon}^{+}\ni k\to 0}\Im
  \inp{f,G_kf}\geq \|(H-\i\big )^{-1}\chi_1f\|^2.
\end{equation}
 
We conclude that 
\begin{equation}
  \label{eq:37}
  \chi_1f=0.
\end{equation} So $0=G_0^+f=\chi_2R_0^D\chi_2f$, and therefore
\begin{equation}
  \label{eq:38}
  R_0^D\chi_2f=0\text{ on }\supp (\chi_2).
\end{equation} We apply $H^D$ to (\ref{eq:38}) and conclude that 
\begin{equation}
  \label{eq:38b}
 \chi_2f=0,
\end{equation} so indeed $f=0$.

\end{proof}

\subsection{Construction of resolvent} 
Due to Proposition \ref{prop:no_kernel}
we can write, cf. (\ref{eq:26}), 
\begin{equation}
    \label{eq:26bb}
   (H-k^2)G_k(I+ K^{\pm})^{-1}=I+\zeta(k)|\psi_0\rangle \langle \phi^\pm | +O\big (|k|^{s'-1}\big ), 
  \end{equation} 
  for $  k\in \Gamma_{\theta,\epsilon}^{\pm}:\,$, where
  \begin{equation}
    \label{eq:58}
   \phi^{\pm}:=\parb{(I+ K^{\pm})^{-1}}^*\chi_2\phi_0. 
  \end{equation}
 We have
  \begin{equation}
    \label{eq:39}
    \big (I+\zeta(k)|\psi_0\rangle \langle \phi^{\pm}|\big )^{-1}=I-\tfrac{\zeta(k)}{\eta^{\pm}(k)}|\psi_0\rangle \langle \phi^{\pm}|;\,\eta^{\pm}(k):=1+\zeta(k)\inp{\phi^{\pm},\psi_0}.
  \end{equation} Of course this is under the condition that 
  \begin{equation}
    \label{eq:41}
    \eta^{\pm}(k)\neq 0.
  \end{equation}

  \begin{lemma}
    \label{lemma:constr-resolv}
For all $k\in\Gamma_{\theta,\epsilon}^{\pm}$ the condition
(\ref{eq:41}) is fulfilled.
  \end{lemma}
  \begin{proof}
 Let us prove (\ref{eq:41})  for the superscript ``$+$
  case''. The ``$-$ case'' is similar.  

Suppose on the contrary that $\eta^+(k)=0$ for some
$k\in\Gamma_{\theta,\epsilon}^+$. Then 
\begin{equation}
  \label{eq:42}
  k^{2\nu}=\tfrac{\bar D_\nu}{D_\nu}
\frac {\e ^{\sigma \pi}}{1-2\i \sigma \inp{\phi^+,\psi_0}}.
\end{equation} 
$k^{2\nu}$ be oscillatory,  the set of all solutions of (\ref{eq:42}) in
$\Gamma_{\theta,\epsilon}^+$ constitutes a sequence converging to 
zero. In particular we can pick a sequence
$\Gamma_{\theta,\epsilon}^+\ni k_n\to 0$ with $0\neq \eta^+(k_n)\to
0$. We apply (\ref{eq:26bb}) and (\ref{eq:39}) to this sequence $\{k_n\}$. Substituting  (\ref{eq:39}) into  (\ref{eq:26bb}) and
multiplying the equation obtained by $\eta(k_n)$, we get
\[
(H-k_n^2)G_{k_n}(I+ K^{+})^{-1} (\eta^+(k_n) - \zeta(k_n) |\psi_0\rangle \langle \phi^{+}| \;)= \eta^+(k_n) (1+ +O\big (|k_n|^{s'-1}\big )).
\]
Taking the limit $n\to\infty$, this leads to
\begin{equation}
  \label{eq:43}
  -\zeta(\infty)HG_\infty^+(I+ K^{+})^{-1}|\psi_0\rangle \langle \phi^{+}|=0.
\end{equation} 
Here $\zeta(\infty):=\lim_{n\to \infty}\zeta(k_n)$ can be  computed by
substituting $ k^{2\nu}$ given by (\ref{eq:42}) in the expression for
$\zeta(k)$ (this is the limit and one sees that it is nonzero), and similarly for
$G_\infty^+:=\lim_{n\to \infty}G_{k_n}$. We learn that 
\begin{equation}
  \label{eq:43b}
 Hu^+=0;\,u^+:=G_\infty^+f^+,\,f^+:=(I+ K^{+})^{-1}\psi_0.
\end{equation}
 Now, the argument of integration by parts  used in (\ref{eq:33}) applied to $u^+$
leads to
 \begin{equation}
   \label{eq:44}
   0=\sigma \parbb{\big |1-\tfrac{ \zeta(\infty)}{2\nu}\big |^2-\big |\tfrac{ \zeta(\infty)}{2\nu}\big |^2}|\inp{\chi_2\phi_0,f^+}|^2.
 \end{equation} We claim that
 \begin{equation}
   \label{eq:40}
   \inp{\chi_2\phi_0,f^+}=0.
 \end{equation} 
In fact for any  $k\in\Gamma_{\theta,\epsilon}^+$ obeying (\ref{eq:42}),  
\begin{equation}
   \label{eq:45bb}
\big |1-\tfrac{ \zeta(\infty)}{2\nu}\big |^2/\big |\tfrac{ \zeta(\infty)}{2\nu}\big |^2=|\e^{\sigma\pi} k^{-2\nu}|^2 =e^{2\sigma (\pi - 2 \arg k)}>1,
   \end{equation}
 whence indeed (\ref{eq:40}) follows from (\ref{eq:44}).

Using (\ref{eq:40}) we can mimic the rest of the proof of Proposition
\ref{prop:no_kernel} and eventually conclude that $f^+=0$. This is a
contradiction since $\psi_0\neq 0$.

  \end{proof}

Combining (\ref{eq:26bb})--(\ref{eq:41}) we obtain (possibly by taking
$\epsilon>0$ smaller)
\begin{equation}
    \label{eq:26bbcc}
    \forall k\in \Gamma_{\theta,\epsilon}^{\pm}:\,(H-k^2)G_k(I+ K^{\pm})^{-1}\parb{I-\tfrac{\zeta(k)}{\eta^{\pm}(k)}|\psi_0\rangle \langle \phi^{\pm}|}\parb{I+O\big (|k|^{s'-1}\big )}=I.
  \end{equation} In particular we have derived a formula for the
  resolvent.

\subsection{Asymptotics of  resolvent}\label{Asymptotics of  resolvent}

Let $u$ be any nonzero {\it regular} solution to the equation
  \begin{equation}\label{eq:50}
   -u''(r)+ \parb{V_\infty(r)+V(r)}u(r)=0.
  \end{equation}
By  regular solution, we means that  the
  function $r\to \chi(r<1)u(r)$ belongs to $\vD (H)$. It will be shown in   Appendix \ref{Regular  zero energy solutions} that the regular solution $u$ is fixed up to a constant  and can be chosen real-valued. 
See (\ref{eq:49ii}) for a formula and 
  for further elaboration. Let
  \begin{equation}
    \label{eq:53}
    R(k):=(H-k^2)^{-1}\mforall\, k\in \Gamma_{\theta,\epsilon}^{\pm}\cap\{\Im k>0\}.
  \end{equation}

  \begin{thm}
    \label{thm:constr-resolv} There exist (finite) rational
    functions $f^{\pm}$ in the variable  $k^{2\nu}$ for $k\in
    \Gamma_{\theta,\epsilon}^{\pm}$ for which
    \begin{equation}\label{eq:45}
      \lim_{\Gamma_{\theta,\epsilon}^{\pm}\ni k\to
        0}\Im f^{\pm}(k^{2\nu})\text{ do not exist,}
    \end{equation}  there exist  Green's functions for $H$ at zero
    energy, denoted $R_0^{\pm}$,  and there  exists a real nonzero regular
    solution to (\ref{eq:50}), denoted $u$,  such that the following asymptotics hold. For all $s>s'>1,\,s'\leq 3$, there
   exists $C>0$:
 \begin{align}
  \forall\, k&\in \Gamma_{\theta,\epsilon}^{\pm}\cap\{\Im k>0\}:\nonumber\\
  &\big \|(H-\i )\parb{R(k)-R_0^{\pm}-f^{\pm}(k^{2\nu})
  |u\rangle \langle u|}\big \|_{\vB(\vH_s,\vH_{-s})} \leq
C|k|^{s'-1}.\label{eq:24bbbb}
\end{align} Here 
\begin{equation}
  \label{eq:46}
  \forall s>1:\;(H-\i )R_0^{\pm} =I -\i R_0^{\pm} \in \vB(\vH_s,\vH_{-s}) \mand (H-\i) u=-\i u\in \vH_{-s}.
\end{equation}
\end{thm}
\begin{proof} By (\ref{eq:26bbcc})
  \begin{equation}
    \label{eq:48}
    R(k)=G_k(I+ K^{\pm})^{-1}\parb{I-\tfrac{\zeta(k)}{\eta^{\pm}(k)}|\psi_0\rangle \langle \phi^{\pm}|}\parb{I+O\big (|k|^{s'-1}\big )}
  \end{equation}
  for $k\in\Gamma_{\theta,\epsilon}^{\pm}$. We expand the product yielding up to errors of order $O\big
   (|k|^{s'-1}\big )$
   \begin{subequations}
     \begin{align}
     \label{eq:49}
       R(k)&\approx
     R_0^{\pm}+\tfrac{\zeta(k)}{\eta^{\pm}(k)}|u_1^{\pm}\rangle\langle u_2^{\pm}|; \quad \mbox{ where }\\
 R_0^{\pm}&=G_0^{\pm}(I+ K^{\pm})^{-1},\label{eq:49i}\\
u_1^{\pm}&=- R_0^{\pm}\psi_0+\chi_2\phi_0,\label{eq:49ii}\\
u_2^{\pm}&=\phi^{\pm}=\parb{(I+ K^{\pm})^{-1}}^*\chi_2\phi_0.\label{eq:49iii}
\end{align}
   \end{subequations}
   
Clearly, $u_2^\pm \neq 0$. According to (\ref{eq:26}), $H u_1^\pm =
-\psi_0 + H(\chi_2 \phi_0) =0$. In addition, $u_1^\pm \neq 0$. In
fact  for $r>14$ (ensuring  that $\chi_2 (r) =1$)  one has
\[
u_1^\pm = - R_0^D f^\pm + \phi_0, \quad \mbox{ with } f^\pm = \chi_2 (1+ K^\pm)^{-1}\psi_0 \in \vH_s, \quad s>1.
\]
 Using  then (\ref{eq:19i}) and \eqref{eq:7} we compute
\begin{equation*}
  r^{-1/2-\nu}\parb{r\tfrac{\d}{\d r}-(1/2-\nu)}u_1^\pm(r) =1-\int_r^\infty \tau ^{\frac 1 2 -\nu} f^\pm(\tau) d\tau,
\end{equation*} 
showing  that $u_1^\pm (r)\neq 0$ for all $r$ large enough.  
 By the uniqueness of regular solutions,  there exist constants $b_\pm \neq 0$ such  that $u_1^{\pm} = b^\pm u$, where $u$ is a real-valued nonzero regular solution  to (\ref{eq:50}).
 Combining the duality relation $R(k)^* = R(\overline{k})$ and  (\ref{eq:49}), we obtain that
\begin{equation}
  \label{eq:51}
  u_2^{\pm}=c^{\pm}u_1^{\mp} = c^{\pm}b^\mp u \quad \text{ for some constants } c^{\pm} \neq 0.
\end{equation}
  Whence indeed (\ref{eq:24bbbb}) holds with 
\begin{equation}
  \label{eq:52}
  R_0^{\pm}=G_0^{\pm}(I+ K^{\pm})^{-1}\mand f^{\pm}(k^{2\nu})=C^{\pm}\tfrac{\zeta(k)}{\eta^{\pm}(k)},
\end{equation} 
where the constants $C^{\pm}=  c^\pm b^\mp \overline{b^\pm} $ are nonzero.  Whence  indeed (\ref{eq:45}) holds. The properties
(\ref{eq:46}) follow from the expressions (\ref{eq:49i}) and (\ref{eq:49ii}). 
\end{proof}

  \begin{cor}
    \label{cor:constr-resolv} There is a limiting absorption principle
    at energies in $]0,\epsilon^2]$: 
    \begin{equation}
      \label{eq:54}
      \forall\,k'\in [-\epsilon,\epsilon]\setminus\{0\}\,\forall s>1:\,R(k'):= \lim_{\Gamma_{\theta,\epsilon}\cap\{\Im k>0\}\ni k\to
        k'}R(k)\text{  exists in } \vB(\vH_s,\vH_{-s}). 
    \end{equation}
 In particular
    \begin{equation}
      \label{eq:47}
      ]0,\epsilon^2]\cap
    \sigma_{\pp}(H)=\emptyset.
    \end{equation} Moreover the bounds (\ref{eq:24bbbb}) extend to
    $\Gamma_{\theta,\epsilon}^{\pm}$. 

Introducing the spectral density as an operator in
$\vB(\vH_s,\vH_{-s})$, $s>1$,
\begin{equation*}
\delta (H-k^2):=\frac{R(k)-R(-k)}{2\pi\i} \mfor 0<k\leq \epsilon,
\end{equation*} we have
\begin{equation}\label{eq:45b}
    \lim_{k\searrow 0}\delta (H-k^2)\text{ does not
        exist}.
    \end{equation} 
  \end{cor}
  \begin{proof}
    Only (\ref{eq:45b}) needs a comment: We represent $R(-k)=R(k)^*$
    and use (\ref{eq:24bbbb}) yielding 
\begin{equation*}
       \delta (H-k^2)\approx (2\pi\i)^{-1}\parb{R_0^{+}-\parb{R_0^{+}}^*}+\tfrac{\Im f^{+}(k^{2\nu})}{\pi} |u\rangle \langle u|.
\end{equation*} The right hand side does not converge, cf.  (\ref{eq:45}).
  \end{proof}
  
 \subsection{$d$-dimensional Schr\"odinger operator}  As another application of Theorem \ref{thm:constr-resolv}, we consider a $d$-dimensional Schr\"odinger operator with spherically symmetric potential of the form
 \[
  H= -\Delta  + W(|x|) \]
 in $L^2(\R^d)$, $d\ge 2$, where $W$ is continuous  and
 $W(|x|) =  -\frac{\gamma}{|x|^2} $ for $x$ outside some compact set
 and   $\gamma> (\frac d 2 -1)^2$. Assume that
 \begin{equation}\label{eq:gamma}
 \gamma \neq (l + \tfrac d 2-1)^2, \quad   l \in \N.
 \end{equation}
 Denote  $\N_\gamma = \{ l \in \N\cup\{0\}|  ( l + \tfrac d 2 - 1)^2 < \gamma\}$.
 Let $\pi_l$ denote the spectral projection associated to the eigenvalue $l(l + d-2)$, $l\in \N\cup\{0\}$,  of the Laplace-Beltrami operator on $\bS^{d-1}$  (and also its natural extension as operator on  $\bH=L^2(\R^d)$).
Then $H$ can be decomposed into a direct sum
\[
H = \oplus_{l=0}^\infty \widetilde H_l \pi_l,
\]
where 
\[
\widetilde H_l = -\frac{d^2}{dr^2} - \frac{d-1}{r} \frac{d}{dr}  + \frac{l(l + d-2)}{r^2} + W(r)
\]
on $\widetilde \vH:=L^2(\R_+; r^{d-1} dr)$.  
When $l\in \N_\gamma$, we can  apply  Theorem \ref{thm:constr-resolv}
with $\nu = \nu_l$, $\nu_l^2 =  ( l + \frac d 2 - 1)^2 - \gamma <0$,
to expand the resolvent $ (\widetilde H_l -k^2)^{-1}$ up to
$O(|k|^\epsilon)$ (see Section \ref{Asymptotics of physical phase
  shift for a potential} for a relevant reduction of $\widetilde H_l$
used here). For $l \not\in \N_\gamma$, the resolvent $(\widetilde
H_l-k^2)^{-1}$ may have singularities at zero, according to whether
zero is an eigenvalue and/or a resonance of $\widetilde H_l$ (defined below). 

Denote  $\widetilde \vH_s=\inp{r}^{-s} \widetilde \vH$ and $
\bH_s=\inp{x}^{-s} \bH$, $s\in \R$. Under the condition
(\ref{eq:gamma}), we say that $0$ is a {\it resonance} of $H$ if there
exists $u \in \bH_{-1}\setminus \bH$ such that $Hu=0$.  We call such
function $u$ a {\it resonance function}. (If the
condition (\ref{eq:gamma}) is not satisfied, the definition of zero
resonance has  to be modified.) The number 
 $0$ is called a {\it regular
point} of $H$ if it is neither an eigenvalue nor a resonance of
$H$. The same definitions apply for $\widetilde H_l$ on
$\widetilde \vH$.  
  Clearly Lemma \ref{lemma_bound} stated below shows that
  for any resonance function $u$ necessarily 
$\pi_l u=0$ for all $l\in \N_\gamma$.  In fact  Lemma
\ref{lemma_bound}   shows that $0$ is a regular point
of $ \widetilde H_l$, $l\in \N_\gamma$.

If $Hu =0$  and $u \in
\bH_{-1}$,  then  by expanding $u$ in spherical harmonics, one can show that (cf.
Theorem 4.1 of \cite{Wan1})
\begin{subequations}
  \begin{equation} \label{expansion}
u(r\theta) =  \frac{\psi(\theta)}{ r^{\frac{d-2}{2} +\mu} } + v, 
\end{equation}
where $v\in L^2(|x|>1)$,
\begin{eqnarray}\label{eq:99}
\mu &= &\sqrt{ ( m + \frac d 2 -1)^2 -\gamma}, \quad  m  =\min \N \setminus \bN_\gamma,  \label{mu} \\
\psi(\theta) &= &
\sum_{j = 1}^{n_\mu}
 -\frac{1}{2\mu} \w{ (W + \frac{\gamma}{|y|^2})u, |y|^{-\frac{d-2}{2} + \mu}\varphi_\mu^{(j)}}\varphi_\mu^{(j)}(\theta), \label{psi}
\end{eqnarray}
 Here $\{ \varphi_\mu^{(j)}, 1 \le j \le n_\mu\}$ is an orthonormal basis of the eigenspace of $-\Delta_{\bS^{d-1}}$ with eigenvalue $m(m+ d-2)$ and $n_\mu$ its multiplicity (cf. \cite{Ta2}):
\begin{equation}
n_\mu = \frac{ (m+d-3)!}{(d-2)! (m-1)!}  + \frac{ (m+d-2)!}{(d-2)! m!}.
\end{equation} 
\end{subequations}

The expansion (\ref{expansion}) implies that a solution $u$ to
  $Hu=0$ with $u \in \bH_{-1}$ is a resonance function  of $H$ if and only if $\mu\in ]0,1]$ and $\psi \neq 0$ and that if zero is a resonance, its multiplicity (cf. \cite{JN, Wan1} for the definition) is at most $n_\mu$. 
Conversely, if  the equation $H u=0$ has  a solution $u\in\bH_{-1}\setminus \bH$,  then the equation
\[
\widetilde H_m g =0
\]  
has a nonzero regular solution $g \in \widetilde \vH_{-1}$ decaying like $ \frac{1}{r^{\frac{d-2}{2} +\mu} }$ at infinity. It  follows that
$u_j = g\otimes \varphi_\mu^{(j)}$,  $1 \le j \le n_\mu$,  are all
resonance functions of $H$. This proves that if $0$ is a resonance of
$H$  its multiplicity is equal to $n_\mu$.

Now let us come back to the asymptotics of the resolvent $R(k) =
(H-k^2)^{-1}$ near $0$.  If $0$ is a regular point of $H$ (this is a
generic  condition and concerns by the discussion above only  sectors $\widetilde H_l$ with $l \not\in \N_\gamma$),
then it is a regular point for all $\widetilde H_l$ with $l \not\in
\N_\gamma$.  One
deduces easily that there exists $R_0^{(l)} \in \vB(\widetilde \vH_s,
\widetilde \vH_{-s})$ for all $s>1$, such that  for any such $s$ there  exists $\epsilon >0$:
\begin{equation} \label{eq:3.45}
  (\widetilde H_l -k^2)^{-1} = R_0^{(l)} + O_l(|k|^\epsilon)\text{ in
  }\vB(\widetilde \vH_s, \widetilde \vH_{-s}) \text{ for $|k|$ small  and $k^2\not \in [0,\infty[$}. 
\end{equation} 
The error term can be uniformly estimated in $l$ as in \cite{Wan1}, yielding an expansion for  $R(k)$.
If $0$ is a resonance but not an eigenvalue of $H$, then $0$ is a regular point for all $\widetilde H_l$ with $l  \not\in \N_\gamma \cup \{m\}$ and the expansion (\ref{eq:3.45}) remains  valid for such $l$.   When $l=m$, $(\widetilde H_m -k^2)^{-1} $ contains a singularity at $0$ which can be calculated as in \cite{Wan2}. Let
\begin{equation}
k_\mu =\left\{ \begin{array}{ccl}
k^{2\mu}, &  \quad & \mbox{ if } \mu  \in ]0, 1[ \\
  k^2 \ln (k^2), & & \mbox{ if } \mu =1.
 \end{array}
 \right.
\end{equation}
Then there exist $g \in \widetilde \vH_{-1}\setminus \widetilde \vH $
verifying $ \widetilde H_m g =0 $, a rank-one operator-valued entire
function $\zeta \to F_m(\zeta) \in \vB(\widetilde \vH_s, \widetilde
\vH_{-s})$, $s>1$, verifying $F_m(0)
=0$   and   $R_0^{(m)} \in \vB(\widetilde \vH_s, \widetilde
\vH_{-s})$,  $s>3$,
such  that for any $s>3$
\begin{equation} \label{Hm} 
(\widetilde H_m -k^2)^{-1} = \frac{e^{i\mu \pi}}{k_\mu} |g\rangle \langle g|+   \frac{1}{k_\mu} F_m( \frac{k^2}{k_\mu}) + R_0^{(m)} + O(\frac{|k|^2}{|k_\mu|})\text{ in }\vB(\widetilde \vH_s,\widetilde \vH_{-s}).  
\end{equation} 
  In particular if  $\mu\in]0, \frac 1
2]$   one has
\begin{equation*}
(\widetilde H_m -k^2)^{-1} = \frac{e^{i\mu \pi}}{k_\mu} |g\rangle
\langle g| + R_0^{(m)} + O(|k|)\text{ in }\vB(\widetilde \vH_s, \widetilde \vH_{-s}),s>3,
\end{equation*} while in the ``worse case'', $\mu=1$, the error term
in \eqref{Hm} is of order $O ({|\ln k|^{-1}})$.

Summing up   we have proved the following

\begin{thm}\label{thm3.7}  Assume that $W(|x|)$ is continuous and $W(|x|)
  = -\frac{\gamma}{|x|^2}$ outside some compact set   with  $\gamma> (\frac d 2 -1)^2$
  satisfying (\ref{eq:gamma}).
\begin{enumerate}[i)]
\item \label{item:7} Suppose  that zero is a regular point of
  $H$. Then there exist  $R_0^\pm \in \vB(\bH_s,\bH_{-s})$ and  $v_l \in
  \widetilde \vH_{-s}\setminus \{0\}$ for all $s>1$ and  $l \in
  \N_\gamma$,  such that   for any $s>1$ there exists $\epsilon >0$: 
\begin{align}
   R(k) = &\sum_{l\in \N_\gamma} f_l^{\pm}(k^{2\nu_l})
  ( \; |v_l\rangle \langle v_l| \; ) \otimes \pi_l\nonumber\\& + R_0^{\pm}  +
  O(|k|^{\epsilon})\text{ in }\vB(\bH_s,\bH_{-s})\mfor     k\in \Gamma_{\theta}^{\pm}.
\end{align}  
Here $f_l^\pm(k^{2\nu_l})$ are the oscillatory functions given in Theorem \ref{thm:constr-resolv} with
$\nu= \nu_l = -\i \sqrt{\gamma- ( l + \tfrac d 2 - 1)^2},  \; l \in \N_\gamma $.

\item \label{item:8} Suppose that zero is a resonance of $H$.   Let $m$ and $\mu$ be defined by  (\ref{mu}). Then $\mu \in ]0,1]$ and  the multiplicity of the zero resonance of $H$ is equal to 
\[
 \frac{ (m+d-3)!}{(d-2)! (m-1)!}  + \frac{ (m+d-2)!}{(d-2)! m!}.
 \]
 Suppose in addition that zero is not an eigenvalue of $H$.  Then
 there exist  $ g \in \widetilde \vH_{-1}\setminus \widetilde \vH$
 with $\widetilde H_{m} g =0$, a rank-one operator-valued analytic
 function $\zeta \to F_m(\zeta) \in \vB(\widetilde \vH_s, \widetilde
 \vH_{-s})$, $s>1$, defined for $\zeta $ near $0$  verifying $F_m(0)
 =0$,    and  $R_1^\pm \in \vB(\bH_s,\bH_{-s})$, $s>3$,  such that for
 any  $s>3$ 
\begin{align}
  R(k)  = &  \left(\frac{e^{i\mu \pi}}{k_\mu}  |g\rangle \langle g| + \frac{1}{k_\mu} F_m( \frac{k^2}{k_\mu}) \right )\otimes \pi_m \nonumber \\ 
 & +  \; \sum_{l\in \N_\gamma} f_l^{\pm}(k^{2\nu_l})( \; |v_l\rangle
 \langle v_l| \; ) \otimes \pi_l \nonumber\\
&+ R_1^{\pm}   +  O ({|\ln k|^{-1}})  \text{ in }\vB(\bH_s,\bH_{-s})\mfor     k\in \Gamma_{\theta}^{\pm}.
\end{align} 
Here $f_l^{\pm}$ and $v_l$  are the same as in \ref{item:7}.
\end{enumerate}
\end{thm}

The case that  $0$ is an eigenvalue of $H$ can be studied in a similar way. The zero eigenfunctions of $H$ may have several  angular momenta $l >m$ and the asymptotics of $R(k)$ up to $o(1)$ as $k \to 0$ contains many terms and we do not give details here.  Note that if (\ref{eq:gamma}) is not satisfied and $\gamma =  ( l + \tfrac d 2 - 1)^2$ for some $l\in \N\cup\{0\}$, $(\widetilde H_l -k^2)^{-1}$ may contain a term of the
order $\ln k$ as $k \to 0$. 

\bigskip

\section{Asymptotics for full Hamiltonian, more general
  perturbation} \label{Asymptotics for full Hamiltonian,  general
  perturbation}

We shall ``solve'' the equation 
\begin{equation}\label{eq:sch11bb}
 -u''(r) +\parb{V_\infty(r)+V(r)}u(r)=0
\end{equation} on  the interval $I=]0,\infty[$ for a class of potentials $V$ with faster decay than
$V_\infty$ at infinity (recall
$V_\infty(r)=\tfrac{\nu^2-1/4}{r^2}\chi(r>1)$). In particular we shall
show absence of zero eigenvalue  for a more
 general class of perturbations than prescribed by Condition
\ref{cond:asympt-full-hamilt}. Explicitly we keep Conditions
\ref{cond:asympt-full-hamilt} \ref{item:1} and \ref{item:3} but modify
 Condition
\ref{cond:asympt-full-hamilt} \ref{item:2} as 
\begin{enumerate}[label=2)']
  \item $V(r)=O(r^{-2-\epsilon}),\;\epsilon>0.$
\end{enumerate} This means that we now impose

\begin{cond}\label{cond:asympt-full-hamilt2}
 \begin{enumerate}[1)]
  \item \label{item:1b}
  $V\in C(]0,\infty[,\R)$,
\item \label{item:2b} $V(r)=O(r^{-2-\epsilon}),\;\epsilon>0$,
\item \label{item:3b} $\exists C_1,C_2>0\,\exists \kappa>0:
  C_1(r^{-2}+1)\geq V(r)\geq (\kappa^2-1/4)r^{-2}-C_2$.
  \end{enumerate} 
\end{cond}
  
\begin{lemma}
  \label{lemma_bound} Under  Condition
  \ref{cond:asympt-full-hamilt2} suppose $u$ is a distributional solution to
  \eqref{eq:sch11bb}  obeying one of the following two conditions:
\begin{enumerate}[1)]

\item \label{item:5}
    $u\in L^2_{-1}\text{ (at infinity)}.$
\item \label{item:6} $u(r)/\sqrt r\to 0\mand u'(r)\sqrt r\to 0\mfor r\to \infty.$
\end{enumerate} Then
  \begin{equation}
    \label{eq:57b}
    u=0.
  \end{equation}
\end{lemma}
\begin{proof} Let $\phi^\pm(r)=r^{1/2\pm \nu}$. Then $\phi^\pm$ 
are linear independent solutions  to the equation 
\begin{equation}\label{eq:sch11}
 -u''(r) +V_\infty(r)u(r)=0;\;r>2.
\end{equation} 
First we shall show that 
\begin{equation}
    \label{eq:56b}
    u=O(r^{1/2-\epsilon})\mand u'=O(r^{-1/2-\epsilon}).
  \end{equation}
 Note that under the condition  \ref{item:5}  in fact $u'\in L^2$ (at
 infinity) due to a standard ellipticity argument.

  We shall  apply  the method of
variation of parameters. Specifically, introduce
``coefficients'' $a_2^+$ and $a_2^-$ of  the ansatz 
\begin{equation}
  \label{eq:ansa}
 u=a^+\phi^++a^-\phi^- .
\end{equation} Using
the differential equations  for $a^+$ and $a^-$ we shall  derive estimates of these
quantities. 

The equations  read 
\begin{equation}\label{kap1}
\left(\begin{array}{cc}\phi^+&\phi^-\\
\frac{\d}{\d\tau}\phi^+&\frac{\d}{\d\tau}\phi^-
\end{array}\right)\frac{\d}{\d\tau}
\left(\begin{array}{c}a^+\\a^-\end{array}\right)
=V
\left(\begin{array}{cc}0&0\\
\phi^+&\phi^-
\end{array}\right)
\left(\begin{array}{c}a^+\\a^-\end{array}\right)
.\end{equation}
 Note that the Wronskian
$W(\phi^-,\phi^+)=\phi^-\tfrac{\d}{\d r}\phi^{+}-\phi^+\tfrac{\d}{\d r}\phi^{-} = 2 \nu$.
 \eqref{kap1} can be transformed into
\begin{equation*} 
 \dfrac{\d}{\d r}\Big ({a^+ \atop a^-}\Big )=N \Big ({a^+ \atop a^-}\Big ),
\end{equation*} where  
\begin{displaymath} 
 N= \frac {V}{2 \nu}\left ( \begin{array}{cc}
\phi^-\phi^+ &(\phi^-)^2\\
-(\phi^+)^2&-\phi^-\phi^+ 
\end{array}\right ).
\end{displaymath}  

Clearly for $V$ obeying Condition \ref{cond:asympt-full-hamilt2}  the quantity $N=O(r^{-1-\epsilon})$ and whence it 
can be integrated to infinity. Whence
there exist
\begin{equation*}
  a^{\pm}(\infty)=\lim_{r\to \infty}a^{\pm}(r);
\end{equation*} in fact
\begin{equation}
  \label{eq:55}
  a^{\pm}(\infty)-a^{\pm}(r)=O(r^{-\epsilon}).
\end{equation}

We need to show that
\begin{equation}
  \label{eq:58p}
   a^{\pm}(\infty)=0.
\end{equation} Note that 
\begin{equation}\label{kap1b}
\left(\begin{array}{cc}\phi^+&\phi^-\\
\frac{\d}{\d\tau}\phi^+&\frac{\d}{\d\tau}\phi^-
\end{array}\right)
\left(\begin{array}{c}a^+\\a^-\end{array}\right)
=
\left(\begin{array}{c}u\\u'\end{array}\right).\end{equation} 
We solve for $(a^+,a^-)$ and multiply the result by
$r^{-1/2}$. Under the condition  \ref{item:5}  each component of the right hand side of the resulting
equation is in $L^2$. Whence also $a^\pm(\infty)/r^{1/2}\in L^2$
and \eqref{eq:58p} and therefore \eqref{eq:56b} follow. We argue
similarly under the condition  \ref{item:6}. 

To show \eqref{eq:57b}  note that the  considerations preceding
\eqref{eq:58p}  hold for all solutions distributional $u$ (not only a 
solution $u$ obeying \ref{item:5} or 
\ref{item:6}) yielding without \ref{item:5} nor 
\ref{item:6}  the
bounds \eqref{eq:56b} with $\epsilon=0$. In particular for a solution $\tilde u$ with
$W(u,\tilde u)=1$ (assuming conversely that $u\neq 0$) we have 
\begin{equation*}
  \int_1^r\,W(u,\tilde u)(x)x^{-1}\,\d x=\ln
    r.
\end{equation*} The right hand side diverges while the left hand side
converges due to \eqref{eq:56b}, and \eqref{eq:57b} follows. 
\end{proof}

Using Lemma \ref{lemma_bound} we can mimic Section \ref{Asymptotics
  for full Hamiltonian, compactly supported perturbation} and obtain
similar results for $H=-\tfrac{\d^2}{\d r^2}+V_\infty+V$ with $V$
satisfying  (the more general) Condition \ref{cond:asympt-full-hamilt2}. In
particular Theorem \ref{thm:constr-resolv} and Corollary
\ref{cor:constr-resolv} hold under 
Condition \ref{cond:asympt-full-hamilt2} provided that we in Theorem
\ref{thm:constr-resolv} impose the additional condition
\begin{equation}
  \label{eq:57}
  s\leq 1+\epsilon/2.
\end{equation} This is here needed to guarantee  that the operators
$K^{\pm}$ of \eqref{eq:29} are  compact on $\vH_s$. Also Theorem 
\ref{thm3.7} has a similar  extension. We leave out further elaboration.

\section{Regular positive energy solutions and asymptotics of phase
  shift}
 \label{Regular positive energy solutions and asymptotics of phase
  shift} Under Condition \ref{cond:asympt-full-hamilt}, or in fact more
generally 
under Condition \ref{cond:asympt-full-hamilt2},  we can define
the notion of regular  positive energy solutions as follows:
Let $k\in \R_+$. A solution $u$ to the equation
  \begin{equation}\label{eq:50h}
   -u''(r)+ \parb{V_\infty(r)+V(r)}u(r)=k^2u(r)
  \end{equation} is  called {\it regular}  if the
  function $r\to \chi(r<1)u(r)$ belongs to $\vD (H)$.  Notice that
  this definition naturally extends the one applied in Section
  \ref{Asymptotics for full Hamiltonian, compactly supported
    perturbation} in the case $k=0$. Again we claim that the regular solution
  $u$ is fixed up to a constant (and hence in particular can be taken
  real-valued): For the uniqueness we may proceed exactly as in
  Appendix \ref{Regular zero energy solutions} (uniqueness at zero energy). For the existence part
  we use the zero energy Green's function $R_0^+$ and the regular zero energy solution
  $u$ appearing in Theorem \ref{thm:constr-resolv}. Consider the
  equation
  \begin{equation}
    \label{eq:59}
    u_{k^2}=u+k^2R_0^+\chi(\cdot<1) u_{k^2}.
  \end{equation} Notice that a solution to \eqref{eq:59} indeed is a
  solution to \eqref{eq:50h} for $r<1$ and hence it can be extended to
  a global solution $\tilde u_{k^2}$. Clearly $\chi(\cdot<1)\tilde
  u_{k^2}\in\vD (H)$ so $\tilde
  u_{k^2}$ is a regular solution. It remains to solve
  \eqref{eq:59} for some nonzero $u_{k^2}$. For that we let $K=R_0^+\chi(\cdot<1) $ and note
  that $K$ is compact on $\vH_{-s}$ for any $s>1$. Whence we have
  \begin{equation}
    \label{eq:60}
    u_{k^2}=(I-k^2K)^{-1}u,
  \end{equation} provided that
  \begin{equation}
    \label{eq:61}
    \Ker (I-k^2K)=\{0\}.
  \end{equation} We are left with showing \eqref{eq:61}. So suppose
  $u_0=k^2Ku_0$ for some $u_0\in \cap_{s>1}\vH_{-s}$,  then we need to
  show that $u_0=0$. Notice that $(H-k^2\chi(\cdot<1))u_0=0$ and that  here
  the second term  can be absorbed into  the potential $V$. The
  computation \eqref{eq:33} shows that also in the present context 
  \begin{equation}
    \label{eq:62}
    0=
\lim_{r\to \infty}\Im \parb{\bar u_0 u'_0}(r) =
\lim_{r\to \infty}\Im \parb{(1/2-\nu)|u_0|^2(r)/r}.
  \end{equation} From \eqref{eq:62} we deduce the condition Lemma \ref{lemma_bound}
  \ref{item:5}  with $u\to
  u_0$ and whence from the conclusion of Lemma \ref{lemma_bound}
  that indeed $u_0=0$.
 
Now let $u_{k^2}$ denote any nonzero real regular solution. By using the
  variation of parameters formula, more specifically by replacing the
  functions $\phi^{\pm}$ in the proof of Lemma \ref{lemma_bound} by
  $\cos(k\cdot)$ and $\sin(k\cdot)$  and repeating the proof (see Step
  I of the proof of Theorem \ref{thm:phase} stated below for details), we find
  the asymptotics
  \begin{equation}
    \label{eq:63}
    \lim_{r\to\infty}
\left(u_{k^2}(r)-C 
\sin \big(k r+\sigma^{\sr}\big
)\right)
=0.
  \end{equation} Here $C=C(k) \neq 0$. Assuming (without loss) that $C>0$  the (real) constant
  $\sigma^{\sr}=\sigma^{\sr}(k)$ is determined modulo $2\pi$. 
  \begin{defn}
    \label{defn:phase shift} The quantity $\sigma^{\sr}=\sigma^{\sr}(k)$
    introduced above is called the {\it phase shift} at energy $k^2$.
  \end{defn}
\begin{defn}
    \label{defn:per phase shift} The notation 
    $\sigma^{\per}=\sigma^{\per}(t)$ signifies  the continuous real-valued $2\pi$-periodic function
    determined by
\begin{equation}\label{eq:def_per}
\begin{cases}
\sigma^{\per}(0)&=0\\
\e^{\pi\sigma}\e^{-\i t}-\e^{\i t}&=r(t)\e^{\i (\sigma^{\per}(t)-t)};\;\;t\in \R,\;r(t)>0
\end{cases}\;.
\end{equation}
\end{defn}

  \begin{thm}
    \label{thm:phase} Suppose Condition
    \ref{cond:asympt-full-hamilt2}. The phase shift $\sigma^{\sr}(k)$ can be chosen
    continuous in $k\in \R_+$. Any such choice  obeys the following asymptotics as
    $k\downarrow 0$:  There exist $C_1,C_2\in \R$ such that
    \begin{equation}
      \label{eq:64}
     \sigma^{\sr}(k)+\sigma\ln k- \sigma^{\per}( \sigma\ln k+C_1)\to 
C_2\mfor k\downarrow 0.    
\end{equation}
\end{thm}
\begin{proof}
\noindent {\bf Step I} We shall show the continuity. From
\eqref{eq:59} and \eqref{eq:60} we see that for any $r>0$ the functions
$]0,\infty[\ni k\rightarrow u_{k^2}(r)$ and $]0,\infty[\ni
k\rightarrow u_{k^2}'(r)$ are continuous. Similar statements hold upon
replacing 
$u_{k^2}\rightarrow \Re u_{k^2}$ and $u_{k^2}\rightarrow \Im u_{k^2}$
which are both real-valued regular solutions (solving \eqref{eq:50h} for $r<1$). Since $ u_{k^2}\neq 0$ one of
these functions  must be nonzero. Without loss we can assume
that $ u_{k^2}$ is a real-valued nonzero regular solution obeying 
that for
$r=1/2$  the functions
$]0,\infty[\ni k\rightarrow u_{k^2}(r)$ and $]0,\infty[\ni
k\rightarrow u_{k^2}'(r)$ are continuous. By a standard regularity
result for linear ODE's with continuous coefficients these  results then hold  for any $r>0$
too. Moreover  (to used in Step II) we have (again for $r>0$ fixed)
\begin{subequations}
  \begin{align}
  \label{eq:65}
  u_{k^2}(r)&- u(r)=O(k^2) \mfor k\downarrow 0\\
 u_{k^2}'(r)&- u'(r)=O(k^2) \mfor k\downarrow 0.\label{eq:66}
\end{align}
\end{subequations}

We introduce
\begin{equation}
  \label{eq:67}
  \phi^+(r)=\cos kr\mand  \phi^-(r)=\sin kr.
\end{equation} Mimicking the proof of Lemma \ref{lemma_bound} we write 
\begin{equation}
  \label{eq:ansa2}
 u_{k^2}=a^+\phi^++a^-\phi^-.
\end{equation} Noting that the Wronskian
$W(\phi^-,\phi^+)=-k$
we have 
\begin{equation} \label{eq:69}
 \dfrac{\d}{\d r}\Big ({a^+ \atop a^-}\Big )=N \Big ({a^+ \atop a^-}\Big ),
\end{equation} where  
\begin{align} \label{eq:70}
 N= -k^{-1}(V_\infty+V)\left ( \begin{array}{cc}
\phi^-\phi^+ &(\phi^-)^2\\
-(\phi^+)^2&-\phi^-\phi^+ 
\end{array}\right ).
\end{align} Since $N=O(r^{-2})$ there exist
\begin{equation}\label{eq:71}
  a^{\pm}(\infty)=\lim_{r\to \infty}a^{\pm}(r).
\end{equation} By the same argument as before either $a^{+}(\infty)\neq
0$ or $a^{-}(\infty)\neq 0$. We write
\begin{equation}
  \label{eq:68}
  (a^{+}(\infty),a^{-}(\infty))/\sqrt{a^{+}(\infty)^2+a^{-}(\infty)^2}=(
\sin
\sigma^{\sr},\cos \sigma^{\sr})
    \end{equation} and conclude the asymptotics \eqref{eq:63} with some
$C\neq 0$. It remains to see  that $a^{\pm}(\infty)$ are continuous in
$k$ (then by \eqref{eq:68} $\sigma^{\sr}$ can be chosen  continuous
too). For that we use the ``connection formula''  
\begin{displaymath} 
 \Big ({u_{k^2} \atop u_{k^2}'}\Big )= \left ( \begin{array}{cc}
\phi^+ &\phi^-\\
{\phi^+}'&{\phi^-}' 
\end{array}\right )\Big ({a^+ \atop a^-}\Big )
\end{displaymath} which is ``solved'' by
\begin{align} \label{eq:65p}
  \Big ({a^+ \atop a^-}\Big )= -k^{-1}\left ( \begin{array}{cc}
      {\phi^-}'&-\phi^-\\
      -{\phi^+}'&\phi^+ 
\end{array}\right )\Big ({u_{k^2} \atop u_{k^2}'}\Big ).
\end{align} We use \eqref{eq:65p} at $r=1/2$. By the comments at the
beginning of the proof the right hand side is  continuous in
$k$ and therefore so is the left hand side. Solving \eqref{eq:69} by integrating from $r=1/2$ and noting that
\eqref{eq:70} is continuous in
$k$ we then conclude that $a^\pm(r)$ are continuous in
$k$ for any $r>1/2$. Since the limits \eqref{eq:71} are taken locally uniformly in
$k>0$ we consequently deduce that indeed $a^{\pm}(\infty)$ are continuous in
$k$.

\noindent {\bf Step II} We shall show \eqref{eq:64} under Condition
\ref{cond:asympt-full-hamilt}.
We shall mimic Step I with \eqref{eq:67} replaced by 
\begin{equation}
  \label{eq:67i}
  \phi^+(r)=r^{1/2}H_\nu^{(1)}(kr)\mand  \phi^-(r)=r^{1/2}\overline {H_{-\nu}^{(1)}(kr)}.
\end{equation} For completeness of presentation note  that in terms of another  Hankel function, 
cf. \cite[(3.6.31)]{Ta1}, 
$\phi^-(r)=r^{1/2}{H_{\nu}^{(2)}(kr)}$. We compute the Wronskian
$W(\phi^-,\phi^+)=4\i/\pi$, cf. \eqref{eq:3} and \cite[(3.6.27)]{Ta1}. Since $V(r)=0$ for
$r\geq R$
\begin{equation}
  \label{eq:72}
  a^{\pm}(r)=a^{\pm}(\infty)\mfor r\geq R.
\end{equation} Moreover \eqref{eq:65p} reads
\begin{align} \label{eq:65pp}
  \Big ({a^+ \atop a^-}\Big )= \tfrac{\pi}{4\i}\left ( \begin{array}{cc}
      {\phi^-}'&-\phi^-\\
      -{\phi^+}'&\phi^+ 
\end{array}\right )\Big ({u_{k^2} \atop u_{k^2}'}\Big ).
\end{align} We will use \eqref{eq:65pp} at $r=R$. Clearly the right
hand side is continuous in
$k>0$  and therefore so is the left hand side. From the asymptotics
\begin{align}
  \label{eq:73}
 \phi^+ (r)-C_\nu \parb{\tfrac {2}{\pi
    k}}^{1/2}\e^{\i kr}&\rightarrow 0\mfor r\rightarrow \infty,\\
\phi^- (r)-\overline{ C_{-\nu}} \parb{\tfrac {2}{\pi
    k}}^{1/2}\e^{-\i kr}&\rightarrow 0\mfor r\rightarrow \infty;\label{eq:74}\\
C_\nu&:=\e^{-\i \pi(2\nu +1)/4},\nonumber
\end{align} we may readily rederive the continuity statement shown
more generally in  Step I. The point is that now we can ``control'' the
limit $k\rightarrow 0$. To see this we need to compute the asymptotics
of the matrix in \eqref{eq:65pp} as  $k\rightarrow 0$ (with
$r=R$). Using \eqref{eq:3} we compute
\begin{subequations}
 \begin{align}
  \label{eq:75}
  \phi^+ (R)&=\frac{ 1}{\i \sin (\nu\pi)}\parbb{ \frac{ 2^\nu
      R^{\frac12 -\nu}}{\Gamma(1-\nu)}k^{-\nu}-\e^{-\sigma\pi}\frac{
      2^{-\nu} R^{\frac12 +\nu}}{\Gamma(1+\nu)}k^{\nu}+O(k^2)},\\
\phi^- (R)&=\frac{ 1}{-\i \sin (\nu\pi)}\parbb{ \frac{ 2^\nu
      R^{\frac12 -\nu}}{\Gamma(1-\nu)}k^{-\nu}-\e^{\sigma\pi}\frac{
      2^{-\nu} R^{\frac12
        +\nu}}{\Gamma(1+\nu)}k^{\nu}+O(k^2)}\label{eq:76},\\
{\phi^+ }'(R)&=\frac{ 1}{\i \sin (\nu\pi)}\parbb{ (2^{-1} -\nu)\frac{ 2^\nu
      R^{-\frac12 -\nu}}{\Gamma(1-\nu)}k^{-\nu}-\e^{-\sigma\pi}( 2^{-1}+\nu)\frac{
      2^{-\nu} R^{-\frac12
        +\nu}}{\Gamma(1+\nu)}k^{\nu}+O(k^2)},\label{eq:77}\\
{\phi^- }'(R)&=\frac{ 1}{-\i \sin (\nu\pi)}\parbb{ (2^{-1} -\nu)\frac{ 2^\nu
      R^{-\frac12 -\nu}}{\Gamma(1-\nu)}k^{-\nu}-\e^{\sigma\pi}(2^{-1} +\nu)\frac{
      2^{-\nu} R^{-\frac12 +\nu}}{\Gamma(1+\nu)}k^{\nu}+O(k^2)}.\label{eq:78}
\end{align} 
\end{subequations}  

 We combine \eqref{eq:65} and \eqref{eq:66} for $r=R$ with \eqref{eq:72}--\eqref{eq:78} and obtain
 \begin{align}
   \label{eq:79}
 u_{k^2}(r)&=  \parb{\tfrac {2}{\pi
    k}}^{1/2}\parbb{ \tfrac{\pi}{4\i}\frac{ C_\nu}{\i \sin
  (\nu\pi)}\parb{\e^{\sigma\pi}\overline{D}k^{\nu}-Dk^{-\nu}}+O(k^2)}\e^{\i
  kr}+{\rm h.c.}+o(r^0);\\
D&:= \frac{ 2^\nu
      R^{\frac12 -\nu}}{\Gamma(1-\nu)} \parbb{\tfrac {2^{-1} -\nu}{R}u(R)-u'(R)}.\nonumber
 \end{align} Here  the term $O(k^2)$ depends on $R$ but not on $r$ and
 the term $o(r^0)$ depends on $k$. The second term, denoted by h.c.,  is given as the
 hermitian (or complex)  conjugate of the first term. Note that $D\neq 0$.

We write $D=|D|\e^{\i \theta_0}$ yielding 
\begin{equation}
  \label{eq:81}
  \e^{\sigma\pi}\overline{D}k^{\nu}-Dk^{-\nu}=|\e^{\sigma\pi}\overline{D}k^{\nu}-Dk^{-\nu}|\e^{\i (
  \sigma^{\per}(\sigma 
\ln k+\theta_0 )-(\sigma 
\ln k+\theta_0 ))}.
\end{equation}
 Next we substitute \eqref{eq:81} into \eqref{eq:79}, use that
 $C_\nu=|C_\nu|\e^{-\i \pi/4}$ and conclude \eqref{eq:64} with
 \begin{equation}
   \label{eq:88}
   C_1=\theta_0  \mand C_2=\pi/4-\theta_0+2\pi p\;\text {for some }p\in\Z.
 \end{equation}

\noindent {\bf Step III} We shall show \eqref{eq:64} under Condition
\ref{cond:asympt-full-hamilt2}. This is done by modifying Step II
using the proof of Step I too. Explicitly using again the functions
$\phi^\pm$ of \eqref{eq:67i} ``the coefficients'' $a^\pm$ need to be
constructed. Since $V$ is not assumed to be compactly supported these
coefficients will now depend on $r$. We first construct them at any
large $R$, this is by the formula \eqref{eq:65pp} (at $r=R$). Then the
modification of \eqref{eq:69}
\begin{equation} \label{eq:69bnn}
 \dfrac{\d}{\d r}\Big ({a^+ \atop a^-}\Big )=N \Big ({a^+ \atop a^-}\Big ),
\end{equation} with   
\begin{align} \label{eq:70nn}
 N= \tfrac{\pi}{4\i}V\left ( \begin{array}{cc}
\phi^-\phi^+ &(\phi^-)^2\\
-(\phi^+)^2&-\phi^-\phi^+ 
\end{array}\right ),
\end{align} is invoked. We integrate to infinity using that
$N=O(r^{-1-\epsilon})$ uniformly in $k>0$. This leads to
\begin{subequations}
 \begin{align}\label{eq:71nn}
 a^{\pm}(r)&= a^{\pm}(\infty)+O(r^{-\epsilon}),\\
a^{\pm}(R)&= a^{\pm}(\infty)+O(R^{-\epsilon}),\label{eq:89}
\end{align}  with the error estimates  being uniform in $k>0$.
\end{subequations}  In particular for $r\geq R$
\begin{equation}
  \label{eq:90}
  a^{\pm}(r)=a^{\pm}(R)+O(R^{-\epsilon})+O(r^{-\epsilon})
\end{equation} uniformly in $k>0$. 

From \eqref{eq:90} we obtain the following modification of  \eqref{eq:79}
\begin{align*}
 u_{k^2}(r)&=  \parb{\tfrac {2}{\pi
    k}}^{1/2}\parbb{ \tfrac{\pi}{4\i}\frac{ C_\nu}{\i \sin
  (\nu\pi)}\parb{\e^{\sigma\pi}\overline{D}k^{\nu}-Dk^{-\nu}}+O(k^2)+O(R^{-\epsilon})}\e^{\i
  kr}+{\rm h.c.}+o(r^0);\\
D=D(R)&:= \frac{ 2^\nu
      R^{\frac12 -\nu}}{\Gamma(1-\nu)} \parbb{\tfrac {2^{-1} -\nu}{R}u(R)-u'(R)}.
 \end{align*} The term $O(k^2)$ depends on $R$, and the term
 $O(R^{-\epsilon})$ depends on $k$ but it is estimated uniformly in
 $k>0$.
By Lemma \ref{lemma_bound} the exist $\delta>0$ and a sequence $R_n\to \infty$ such
that
\begin{equation}
  \label{eq:94}
  |D(R_n)|\geq \delta\mforall n.
\end{equation} 
Using these values of $D$  in  \eqref{eq:81} we can write
 \begin{align}
  \label{eq:81hh}
  \e^{\sigma\pi}\overline{D}k^{\nu}-Dk^{-\nu}&=|\e^{\sigma\pi}\overline{D}k^{\nu}-Dk^{-\nu}|\e^{\i (
  \sigma^{\per}(\sigma 
\ln k+\theta )-(\sigma 
\ln k+\theta ))};\\
D&=D(R_n),\;\theta=\theta_n\in [0,2\pi[.\nonumber
\end{align} 
We can assume that for some $\theta_0\in [0,2\pi]$
\begin{equation}
  \label{eq:56}
  \theta_n\to \theta_0\mfor n\to \infty.
\end{equation} Using this number $\theta_0$ we obtain again
\eqref{eq:64} with $C_1$ and $C_2$ given as in \eqref{eq:88}.
\end{proof}

\bigskip

\section{Asymptotics of physical phase shift for a potential like $-\gamma \chi(r>1)r^{-2}$} \label{Asymptotics of physical phase shift for a potential}

We shall reduce a $d$-dimensional
Schr\"odinger equation to angular momentum sectors and discuss the
asymptotics of the ``physical'' phase shift for small angular momenta
in the low energy regime.

We consider for  $d\geq 2$ the stationary $d$-dimensional
Schr\"odinger equation
\begin{equation*}
  Hv=(-\triangle +W)v=\lambda v;\; \lambda >0,
\end{equation*}
 for a radial
potential $W=W(|x|)$ obeying 

\begin{cond}\label{cond:asympt-full-hamilt2b}
 \begin{enumerate}[1)]
\item \label{item:4}$W(r)=W_1(r)+W_2(r);\;W_1(r)=-\tfrac{\gamma}{r^2}\chi(r>1)\mfor\text{
  some }
   \gamma>0$,
  \item \label{item:1bb}
  $W_2\in C(]0,\infty[,\R)$,
\item \label{item:2bb} $\exists\, \epsilon_1,C_1>0:\;|W_2(r)|\leq
  C_1r^{-2-\epsilon_1}\mfor r>1$,
\item \label{item:3bb} $\exists\, \epsilon_2,C_2>0:
  |W_2(r)|\leq C_2r^{\epsilon_2-2}\mfor r\leq  1$.
  \end{enumerate} 
\end{cond}

 Under Condition \ref{cond:asympt-full-hamilt2b} $H=-\triangle +W$ is self-adjoint
 as defined in terms of the Dirichlet form on $H^1(\R^d)$, cf. \cite{DS}. Let  $H_l$,
 $l=0,1,\dots$, be the  corresponding 
reduced  Hamiltonian corresponding to an eigenvalue
$l(l+d-2)$ of
the Laplace-Beltrami operator on $\bS^{d-1}$ 
\begin{equation}
  \label{eq:80}
  H_lu=-u''+(V_\infty+V)u.
\end{equation} Here
\begin{subequations}
  \begin{align}
  \label{eq:83}
 V_\infty(r)&= \tfrac{\nu^2-1/4}{r^2}\chi(r>1);\; \nu^2=(l+\tfrac d2
-1)^2-\gamma,\\
V(r)&=W_2(r)+ \tfrac{(l+\tfrac d2
-1)^2-1/4}{r^2}\big (1-\chi(r>1)\big ),\label{eq:84} 
\end{align}
\end{subequations}
 and  the stationary equation reads
\begin{equation}
  \label{eq:82}
 -u''+(V_\infty+V)u=\lambda u. 
\end{equation} Notice that for
\begin{equation}
  \label{eq:85}
  \gamma>(l+\tfrac d2
-1)^2 ,
\end{equation}
and 
\begin{equation}
  \label{eq:85b}
 (d,l)\neq (2,0),
\end{equation}
 indeed Condition
\ref{cond:asympt-full-hamilt2} is fulfilled and $H_l$ coincides with
the Hamiltonian  given by the
construction of 
Section \ref{Asymptotics for full Hamiltonian,  general
  perturbation}. The case $(d,l)= (2,0)$ needs a separate 
consideration which is given in Appendix \ref{Case (d,l)eq (2,0)}.

 Under the conditions \eqref{eq:85} and \eqref{eq:85b} let  $u_l$ be a regular
solution to the reduced Schr\"odinger equation \eqref{eq:82}. Write 
\begin{equation}
  \lim_{r\to\infty}
\left(u_l(r)-C 
\sin \big(\sqrt \lambda r+D_l\big
)\right)
=0. \label{eq:phase shift22ab1}
\end{equation} The standard definition of the
phase shift (coinciding with the time-depending definition)
 is
 \begin{equation}
   \label{eq:phase shift22b=}
   \sigma_l^\phy(\lambda)=D_l+\tfrac{d-3+2l}{4}\pi.
 \end{equation}
It is known from  \cite{Ya, DS} that for a potential $W(r)$ behaving at
infinity like
$-\gamma r^{-\mu}$   with $\gamma>0$ and $\mu\in
]1,2[$  
\begin{equation}
  \label{eq:relat}
  \exists \sigma_0\in\R: \;\sigma_l^\phy(\lambda)-\int^\infty_{R_0}\Big(
   \sqrt{\lambda}-\sqrt{\lambda-W(r)}\Big )\,\d
   r\to \sigma_0\mfor \lambda \downarrow 0.
\end{equation}
 Here $R_0$ is any sufficiently big positive number, and the integral  does not have a
(finite) limit as $\lambda\downarrow0$. In the present case, $\mu=2$, 
(\ref{eq:relat}) indicates a logarithmic divergence. This is indeed
occurring  although (\ref{eq:relat})  is 
incorrect for $\mu=2$. The correct behaviour of the phase shift under
the conditions \eqref{eq:85} and \eqref{eq:85b} follows directly from Section \ref{Regular positive energy solutions and asymptotics of phase
  shift}:
\begin{thm}
    \label{thm:phasebb} Suppose Condition
    \ref{cond:asympt-full-hamilt2b} and \eqref{eq:85} for some $l\in\N\cup\{0\}$. Let
    \begin{equation}
      \label{eq:98}
      \sigma=\sqrt{\gamma-(l+\tfrac d2
-1)^2}.
    \end{equation}
 The phase shift $\sigma^{\phy}_l(\lambda)$ can be chosen
    continuous in $\lambda\in \R_+$. Any such choice  obeys the following asymptotics as
    $\lambda\downarrow 0$:  There exist $C_1,C_2\in \R$ such that
    \begin{equation}
      \label{eq:64vv}
     \sigma^{\phy}_l(\lambda)+\sigma\ln \sqrt{\lambda}- \sigma^{\per}( \sigma\ln \sqrt{\lambda}+C_1)\to 
C_2\mfor \lambda\downarrow 0.    
\end{equation}
\end{thm}

 Note that we have included the case $(d,l)= (2,0)$ in this
 result. The necessary modifications of Section \ref{Regular positive energy solutions and asymptotics of phase
  shift} for this case are outlined in Appendix \ref{Case (d,l)eq (2,0)}.

\appendix

\bigskip

\section{Regular zero energy solutions} \label{Regular zero energy solutions}
We shall elaborate on the notion of regular solutions  as used in
Sections \ref{Asymptotics for full Hamiltonian, compactly supported
  perturbation} and \ref{Asymptotics for full Hamiltonian,  general
  perturbation}. Recall from the discussion around \eqref{eq:50} that
we call a solution $u$ to \eqref{eq:50} for regular if $r\to
\chi(r<1)u(r)$ belongs to $\vD (H)$ where $H$ is defined in terms of
a  potential $V$ satisfying Condition \ref{cond:asympt-full-hamilt}
(or Condition \ref{cond:asympt-full-hamilt2}). The existence of a
(nonzero) 
regular solution is shown  explicitly by the formula 
  (\ref{eq:49ii}). We shall show that the
  regular solution is 
unique up to a constant. Notice that  as a  consequence of this uniqueness
result a regular solution is real-valued up to  constant. 

Suppose conversely that all solutions are regular. Due to
\cite[Theorem X.6 (a)]{RS} there exists a nonzero solution $v$ to 
\begin{equation}\label{eq:50e}
   -v''(r)+ \parb{V_\infty(r)+V(r)}v(r)=\i v(r)
  \end{equation} which is in $L^2$ at infinity. By the variation of
  parameter formula now
  based on the basis of regular solutions to \eqref{eq:50}, cf. the
  proof of \cite[Theorem X.6 (b)]{RS}, we conclude that $v\in \vD(H)$
  and that $(H-\i)v=0$. This violates that $H$ is self-adjoint.
  
\bigskip

\section{ Case $(d,l)= (2,0)$} \label{Case (d,l)eq (2,0)} 
For $(d,l)= (2,0)$ Condition \ref{cond:asympt-full-hamilt2} fails for
the  operator $H_l$ of Section \ref{Asymptotics of physical phase shift for a potential} (this
example would require $\kappa=0$ in Condition
\ref{cond:asympt-full-hamilt2} \ref{item:3b}). The form domain is not
$H^1_0(\R_+)$ is this case. The form  is given as
follows:
\begin{subequations}
  \begin{align}
  \label{eq:87}
  \vD (Q)&=\{f\in L^2(\R_+)| \,g\in L^2(\R_+)\text{ where
  }g(r)=f'(r)-\tfrac{1}{2r}f(r)\},\\
Q(f)&= \int_0^\infty\parbb{|f'(r)-\tfrac{1}{2r}f(r)|^2+W(r)|f(r)|^2}\,\d r;\;f\in \vD (Q).\label{eq:95}
\end{align}
\end{subequations}

  This is a closed semi-bounded quadratic form and the domain $\vD
  (H)$ of the 
corresponding operator $H$ (cf. \cite{Da, RS}) is characterized  as the subset
of   $f$'s in $\vD (Q)$ for which
\begin{equation}
  \label{eq:96}
  h\in L^2(\R_+)\text{ where }h(r):=\parb
  {-\tfrac{\d ^2}{\d r^2}-\tfrac{1}{4r^2}+W(r)}f(r) 
  \text{ as a distribution on } \R_+,
  \end{equation} and for $f\in  \vD(H)$ we have 
  \begin{equation}
    \label{eq:97}
    (Hf)(r)= \parb {-\tfrac{\d ^2}{\d r^2}-\tfrac{1}{4r^2}+W(r)}f(r).
  \end{equation}

To see the connection to the two-dimensional Hamiltonian  of Section \ref{Asymptotics
  of physical phase shift for a potential} defined with  form domain
$H^1(\R^2)$ let us note the alternative description of $Q$:
\begin{subequations}
  \begin{align}
  \label{eq:91}
  \vD (Q)&=\{f\in L^2(\R_+)| \,\tilde g(|\cdot|)\in H^1(\R^2)\text{ where
  }\tilde g(r)=r^{-1/2}f(r)\},\\
  \label{eq:92}
  Q(f)&=(2\pi)^{-1}\int _{\R^2}\parbb{\big |\nabla \parb{|x|^{-1/2}f(|x|)}\big |^2+W(|x|)\big ||x|^{-1/2}f(|x|)\big |^2}\,\d
  x\mfor f\in \vD (Q).
\end{align} 
\end{subequations} Clearly the integral to the right  in  \eqref{eq:92} is the
form of the two-dimensional Hamiltonian (applied to  radially
symmetric functions).

We also note that $H^1_0(\R_+)\subseteq \vD (Q)$ and that  
\begin{equation*}
  \label{eq:93}
  C_c^\infty(\R_+)+{\rm span} \parb{ f_0};\;f_0(r):=r^{1/2}\chi(r<1),
\end{equation*} is  a core for $Q$. In fact, although $f_0\notin
H^1_0(\R_+)$, the set $C_c^\infty(\R_+)$ is  actually a core for  $Q$. Whence
 $H$ is   the Friedrichs extension of the action \eqref{eq:97}
 on $C_c^\infty(\R_+)$.

Due to \eqref{eq:95} and the description in  \eqref{eq:96} of the domain $\vD
  (H)$ we can show the uniqueness of regular
  solutions exactly as in Appendix \ref{Regular zero energy
    solutions}. The existence of  (nonzero) regular solutions follows
  from 
  the previous scheme too. Indeed the basic operators $K^{\pm}$ of
  \eqref{eq:29} are again compact on $\vB(\vH _s)$. To see this we
  need to see that various terms are compact. Let us here consider the  contribution from the first
  term of \eqref{eq:28} 
  \begin{equation*}
  -\parb{ \chi_1''+2\chi_1'\tfrac{\d}{\d r}}\big (H\mp\i\big )^{-1}\chi_1 +\chi_1(\pm\i)\big (H\mp\i\big )^{-1}\chi_1=: K^{\pm}_1+ K^{\pm}_2.
  \end{equation*} (The contribution from the second term of
  \eqref{eq:28} is treated in the same way as before.) We decompose using any $C>0$ such that $H^0\geq C+1$
  \begin{align*}
 K^{\pm}_1&=B^{\pm}K;\\
B^{\pm} &=-\parbb{ \chi_1''(r)+2\chi_1'(r)\tfrac{1}{2r}+2\chi_1'(r)\parb
  {\tfrac{\d}{\d r}-\tfrac{1}{2r}}}(H\mp\i\big )^{-1}(H-C\big
)^{1/2},\\
K&=(H-C\big
)^{-1/2}\chi_1.
      \end{align*} The operator $B^{\pm}$ is bounded and the operator
      $K$ is compact (the latter may be seen easily by going back to the
      space $L^2(\R^2)$ and there invoking standard Sobolev embedding); whence $ K^{\pm}_1$ is compact. Clearly also $ K^{\pm}_2$ is compact.

\end{document}